\documentclass[a4paper,11pt]{article}
\pdfoutput=1 

\usepackage{jheppub} 

\usepackage[T1]{fontenc} 
\usepackage[nolist,printonlyused]{acronym}
\usepackage{amsthm, amssymb, amsmath}
\usepackage{mathtools}
\usepackage{tikz}
\usetikzlibrary{matrix}
\usepackage{enumitem}
\usepackage{float}
\setlist{nosep}

\usepackage{hyperref}
\usepackage{prettyref}
\definecolor{darkgreen}{rgb}{0,0.5,0}
\hypersetup{final, pdfpagemode=FullScreen, colorlinks=true,citecolor=darkgreen}
\usepackage[hyperpageref]{backref} 
\newcommand{\ZZ}{\mathbb Z}
\newcommand{\QQ}{\mathbb{Q}}
\newcommand{\Qbar}{\overline{\mathbb{Q}}}
\newcommand{\PP}{\mathbb P}
\newcommand{\FF}{\mathbb F}
\newcommand{\CC}{\mathbb C}

\newcommand{\Sbar}{\overline{S}}

\newcommand{\cO}{\mathcal{O}}

\DeclareMathOperator{\Pic}{Pic}
\DeclareMathOperator{\NS}{NS}
\DeclareMathOperator{\Aut}{Aut}

\DeclareMathOperator{\im}{im}
\DeclareMathOperator{\disc}{disc}
\DeclareMathOperator{\rk}{rk}

\DeclareMathOperator{\Br}{Br}
\DeclareMathOperator{\Gal}{Gal}
\DeclareMathOperator{\rank}{rank}
\DeclareMathOperator{\Hom}{Hom}
\DeclareMathOperator{\Triv}{Triv}
\DeclareMathOperator{\MW}{MW}
\DeclareMathOperator{\tors}{tors}
\DeclareMathOperator{\Sym}{Sym}
\DeclareMathOperator{\Spec}{Spec}
\DeclareMathOperator{\Kum}{Kum}
\DeclareMathOperator{\Res}{Res}
\DeclareMathOperator{\Frob}{Frob}

\DeclarePairedDelimiter\abs{\lvert}{\rvert}

\newcommand*{\defeq}{\mathrel{\rlap{%
                     \raisebox{0.3ex}{$\cdot$}}%
                     \raisebox{-0.3ex}{$\cdot$}}%
                     =}

\theoremstyle{plain}
\newtheorem{theorem}{Theorem}[section]
\newtheorem{corollary}[theorem]{Corollary}
\newtheorem{proposition}[theorem]{Proposition}
\newtheorem{lemma}[theorem]{Lemma}

\theoremstyle{definition}

\newtheorem{definition}[theorem]{Definition}
\newtheorem{example}[theorem]{Example}
\newtheorem{remark}[theorem]{Remark}

\title{\boldmath Arithmetic and geometry of a K3 surface emerging from virtual corrections to Drell--Yan scattering}

\author[a,b]{Marco Besier,}
\author[a]{Dino Festi,}
\author[c,d]{Michael Harrison,}
\author[e]{Bartosz Naskr\k{e}cki}

\affiliation[a]{Institut f\"ur Mathematik, Johannes Gutenberg-Universit\"at Mainz, 55099 Mainz, Germany} 
\affiliation[b]{PRISMA Cluster of Excellence, Institut f\"ur Physik, Johannes Gutenberg-Universit\"at Mainz, 55099 Mainz, Germany}
\affiliation[c]{School of Mathematical Sciences, University of Nottingham, NG7 2RD Nottingham, United Kingdom}
\affiliation[d]{MAGMA Computational Algebra Group, School of Mathematics and Statistics, University of Sydney, NSW 2006, Australia}
\affiliation[e]{Faculty of Mathematics and Computer Science, Adam Mickiewicz University, 61-614 Pozna\'{n}, Poland}

\emailAdd{mbesie01@uni-mainz.de}
\emailAdd{dfesti@uni-mainz.de}
\emailAdd{mch1728@gmail.com}
\emailAdd{bartnas@amu.edu.pl}

\abstract{We study a K3 surface, which appears in the two-loop mixed electroweak-quantum chromodynamic virtual corrections to Drell--Yan scattering. 
A detailed analysis of the geometric Picard lattice is presented,
computing its rank and discriminant in two independent ways:
first using explicit divisors on the surface and then using an explicit elliptic fibration.
We also study in detail the elliptic fibrations of the surface and use them to provide an explicit Shioda--Inose structure. 
}
\begin{document} 

\begin{acronym}

\acro{sm}[SM]{Standard Model}

\acro{qcd}[QCD]{quantum chromodynamics}

\acro{ew}[EW]{electroweak}

\acro{qed}[QED]{quantum electrodynamics}

\acro{qft}[QFT]{quantum field theory}

\acro{lhc}[LHC]{Large Hadron Collider}

\acro{gut}[GUT]{Grand Unified Theory}

\acro{mpls}[MPLs]{multiple polylogarithms}

\acro{empls}[eMPLs]{elliptic multiple polylogarithms}

\end{acronym}

\maketitle
\flushbottom
\newpage
\section{Introduction} \label{sec:Introduction}

Given the advancing precision of measurements carried out at modern particle colliders, equally precise theoretical predictions are required.
To perform these computations, one has to solve the most complicated Feynman integrals.
It turns out that the rationality problem for hypersurfaces often marks an essential step in the calculation of these integrals \cite{Besier:2018jen,Chaubey:2019lum,Bourjaily:2018aeq,vonManteuffel:2017hms,Gehrmann:2015bfy,Heller:2019gkq}.
As a consequence, methods from algebraic and arithmetic geometry are becoming increasingly important for theoretical particle physics.

In this paper, we study the rationality problem for a hypersurface derived from Feynman integrals contributing to the mixed electroweak-quantum chromodynamics corrections to Drell--Yan scattering.
The preferred method of solving these Feynman integrals is to solve them in terms of \ac{mpls}, as these functions are well understood and implemented for numerical evaluation \cite{Vollinga:2004sn,Bauer:2000cp}.
To achieve this, one would ideally want to find a rational parametrisation of the projective surface given by 
\begin{equation}\label{eq:DY}
   X_{DY}\colon w^2=4xy^2z(x-z)^2+(x+y)^2(xy+z^2)^2
\end{equation}
in the weighted projective space $\PP(1,1,1,3)$ over $\QQ$ with coordinates $x,y,z,w$ of weights $1,1,1,3$, respectively.
\begin{remark}
For the reader who never encountered weighted projective spaces before,
these can be regarded as a generalisation of projective spaces by arbitrarily changing the \textit{weight} of the coordinates of the space and hence changing the condition for a polynomial to be homogeneous.
For example, in a space in which the coordinates $x_0$ and $x_1$ have weight $1$ and $2$, respectively, the polynomial $x_0^2-x_1$ is homogeneous of degree $2$.
A classical reference for this topic is~\cite{Dol82}.
\end{remark}
\noindent As a first result, we prove the following theorem.
\begin{theorem}\label{t:Main1}
The surface $X_{DY}$ defined in \eqref{eq:DY} is birationally equivalent to a K3 surface. 
Its Picard lattice has rank $19$, discriminant $24$ and discriminant group isomorphic to $\ZZ/2\ZZ \times \ZZ/2\ZZ\times \ZZ/6\ZZ$. The surface $X_{DY}$ admits an explicit Shioda-Inose structure which is related to a classical modular form of level $160$ and weight $2$.
\end{theorem}
\noindent The first two statements of the theorem are proven in Section~\ref{sec:ComputingPicXViaDivisorsOnX}, cf. Proposition~\ref{p:PicSub} and Theorem~\ref{t:LambdaPic};
the third statement is proven in Section~\ref{sec:ComputationOfTheShiodaInoseStructureOfX}, cf. Corollary~\ref{c:Modular} and Theorem~\ref{t:SIStructure}.

\noindent Let us recall the definition of a K3 surface and its Picard lattice and Picard number.
\begin{definition}\label{def:K3}
Let $Y$ be a smooth, projective, geometrically integral surface over a field $k$. 
We say that $Y$ is a K3 surface if it has first cohomology group $H^1(Y,\cO_Y)=0$ and  trivial canonical class $K_Y=0$. 
\end{definition}
\begin{remark}
Alternatively, Definition~\ref{def:K3} is equivalent to saying that a K3
surface is a simply connected Calabi--Yau manifold of dimension $2$,
i.e., a smooth, simply connected surface admitting a nowhere vanishing holomorphic $2$-form.
\end{remark}
\begin{definition}
Let $Y$ be a K3 surface over a field $k$ and let $\overline{k}$ be an algebraic closure of $k$.
With $\overline{Y}$ we denote the change of base $Y\times_{k} \overline{k}$ of $Y$ to $\overline{k}$.
We denote by $\Pic Y$ the {\em Picard lattice} of $Y$ (see~\cite[Chapter 17]{Huy16} for more details);
 $\Pic \overline{Y}$ denotes the {\em geometric} Picard lattice of $Y$, that is, the Picard lattice of $\overline{Y}$.
The {\em Picard number} of $Y$, denoted by $\rho (Y)$, is defined to be the rank of the Picard lattice of $Y$, 
i.e., $\rho (Y)= \rk \Pic Y$;
analogously, the {\em geometric} Picard number of $Y$ is $\rho (\overline{Y})$,
the Picard number of $\overline{Y}$.
\end{definition}

As noted in~\cite{FvS19},
also the two-loop virtual corrections to Bhabha scattering give rise to a K3 surface,
and so one might ask if the surfaces arising from the Bhabha and the Drell--Yan scatterings are related, or even the same.
We prove that this is not the case, by showing that the Picard lattices of the two surfaces have different rank.
\begin{theorem}\label{t:Main2}
Let $X_{DY}$ be the surface defined by \eqref{eq:DY}, 
and let $B$ be the surface defined in~\cite{FvS19}.
Then $X_{DY}$ is neither birationally equivalent nor isogenous to $B$; 
furthermore, it is not birationally equivalent to any of the deformations of $B$ considered in~\cite{FvS19}.
\end{theorem}

Although the surface $X_{DY}$ is not parametrisable by rational functions and it is not isomorphic to the surface arising from Bhabha scattering, something can still be done to solve the integrals:
for example, one can leave the non-rationalisable square root untouched and express the result in terms of \ac{mpls} with algebraic arguments \cite{Heller:2019gkq}.

Alternatively, one may hope to solve the integrals in terms of \ac{empls}. 
This approach has recently led to a very compact result for the master integrals of the two-loop Bhabha corrections \cite{tancredi:2019amp}, and involved elliptic fibrations of the Bhabha K3 surface.
Therefore, we believe that elliptic fibrations on $X_{DY}$ might enable physicists to find a compact result of the Drell--Yan master integrals in terms of \ac{empls}.
For this reason, we present a computational method to find many elliptic fibrations of $X_{DY}$ in Section~\ref{sec:FullClassificationOfEllipticFibrationsOnX}, and explicitly describe three of them.

In order to prove Theorems~\ref{t:Main1} and~\ref{t:Main2}, 
it is enough to consider a smooth model $S_{DY}$ (cf. Definition~\ref{d:SDY}) of $X_{DY}$ and study its (geometric) Picard lattice.
Finding elliptic fibrations on $X_{DY}$ is equivalent to finding elliptic fibrations on $S_{DY}$. 
The methods used are not new,
but this paper represents an attempt to establish an algorithmic and concrete approach to these problems.

We proceed as follows: 
a physical motivation for our results and the proof of Theorem~\ref{t:Main2} (cf. Corollary~\ref{c:NonIso}) are given in Section~\ref{sec:PhysicalBackgroundAndMotivation}.
In Section~\ref{sec:ComputingPicXViaDivisorsOnX} and~\ref{sec:ComputingPicXViaEllipticFibrationsOnX} we compute the Picard lattice of $S_{DY}$ in two different ways: 
exhibiting explicit divisors, and using an elliptic fibration, respectively.
Furthermore, we use the computations in Section~\ref{sec:ComputingPicXViaDivisorsOnX} to deduce some information about the Brauer group of $S_{DY}$ (Subsection~\ref{ssec:ComputationOfTheBrauerGroupOfX}).
The computation of elliptic fibrations of $S_{DY}$ is provided in Section~\ref{sec:FullClassificationOfEllipticFibrationsOnX}.
Besides being useful for (re-)computing the geometric Picard lattice of the K3 surface, 
these elliptic fibrations allow us to explicitly describe a Shioda--Inose structure of $S_{DY}$, 
which is done in Section~\ref{sec:ComputationOfTheShiodaInoseStructureOfX}. 
Consequently, we also compute the number of points on the reduction of the surface $S_{DY}$ to positive characteristic.

Some proofs in this paper are aided by explicit computations using the software package {\tt Magma} (cf.~\cite{Magma}). 
This is explicitly stated in the proofs where such computations are performed.
The code used in the proofs can be found in the ancillary file~\cite{BFHN19} available online.
\section{Physical background and motivation} 
\label{sec:PhysicalBackgroundAndMotivation}
\subsection{Particle physics and Drell--Yan scattering}
In physics, all possible interactions of matter can be reduced to four fundamental forces.
On the one hand, one has gravitational and electromagnetic interactions, whose effects we experience in our everyday life.
On the other hand, one has the strong and the weak interactions that produce forces at subatomic distances.
While the gravitational force is successfully described by Einstein's general theory of relativity, the strong, weak, and electromagnetic interactions are described by the \ac{sm} of particle physics---a term which has become a synonym for a \ac{qft} based on the gauge group $SU(3)\times SU(2)\times U(1)$.
The groups $SU(3), SU(2)$, and $U(1)$ constitute the gauge groups for the strong, weak, and electromagnetic force, respectively.
Accordingly, the \ac{sm} contains three coupling constants $g_1, g_2$ and $g_3$---one for each of the three fundamental interactions described by the \ac{sm}.
The respective \ac{qft}s that are used for the theoretical description of these interactions are \ac{qcd} and \ac{ew} theory, the latter being the unification of weak theory and \ac{qed}.

To test the validity of the \ac{sm}, experimental physicists investigate scattering processes, i.e., collisions of particles generated by electron or proton beams.
In the search for new elementary particles, these collisions are performed at very high energies in huge particle colliders, the world's most famous being the \ac{lhc} at the CERN laboratory in Geneva, Switzerland.

In the regime of high energies, the aforementioned coupling constants $g_1, g_2, g_3$ of the \ac{sm} are very small and perturbation theory, i.e., regarding physical observables as power series in the coupling constants, turns out to be a valuable tool to obtain theoretical predictions.
For this reason, perturbative \ac{qft} is often referred to as theoretical high energy particle physics. 

One of the most critical scattering processes studied at the \ac{lhc} is the Drell--Yan production of $Z$ and $W$ bosons \cite{Drell:1970wh}.
Due to their clean experimental signature, Drell--Yan processes can be measured with comparatively small experimental uncertainty, allowing for very precise tests of the \ac{sm} and numerous applications in other scattering experiments.
For instance, the Drell--Yan mechanism provides valuable information about the parton distribution functions, which are essential for theoretical studies of processes at virtually any hadron collider around the globe.
Because of the sharp experimental signal, Drell--Yan scattering is also used for detector calibration of the \ac{lhc} itself and for the determination of its collider luminosity.
Finally, Drell--Yan processes are crucial in searches for physics beyond the \ac{sm} involving new, yet to discover elementary particles such as $Z^\prime$ and $W^\prime$ that originate from \ac{gut} extensions of the \ac{sm}.
For all these reasons, an accurate and reliable experimental setup as well as very precise theoretical descriptions of the Drell--Yan mechanism are of vital importance for contemporary particle physics at the \ac{lhc}.

Latest theoretical predictions for this scattering process are in reasonable agreement with the experimental data.
Nevertheless, even more precise computations are indispensable.
To improve theoretical accuracy, one needs to take into account higher-order perturbative corrections.
Currently, the theoretical description of Drell--Yan processes includes \ac{qcd} corrections of second order \cite{Altarelli:1979ub,Altarelli:1984pt,Matsuura:1988sm,Hamberg:1990np} as well as \ac{ew} corrections up to first order of the respective perturbation series \cite{Wackeroth:1996hz,Baur:1997wa}.
Second-order corrections to the Drell--Yan process in \ac{qed} with massive fermions were recently considered in \cite{Blumlein:2019srk}.
Moreover, there are some other significant second-order perturbative contributions, whose full analytic structure was also studied only recently, one of the most difficult being the mixed \ac{ew}-\ac{qcd} corrections \cite{Bonciani:2016ypc,Heller:2019gkq,vonManteuffel:2017myy}. 
It is of maximum importance to get a solid understanding of these newly discovered contributions to match future experimental requirements, especially in view of run III of the \ac{lhc}, starting in 2021.

\subsection{Feynman integrals via differential equations}
The crux of a typical computation in theoretical particle physics is the fact that, in order to determine the sought after coefficients of the relevant perturbation series, one has to solve certain integrals, often referred to as Feynman integrals.
For this reason, these integrals may be regarded as the building blocks for the study of any scattering process in perturbative \ac{qft}.
Unfortunately, Feynman integrals are usually extremely difficult to compute and often even divergent under the assumption of a four-dimensional space-time.
In order to deal with these divergences, one needs to introduce a regularisation parameter. 
While there are several ways to do this, the method of dimensional regularisation has become standard. 
Roughly speaking, one replaces a four-dimensional integral by an integral in $D$ dimensions, where $D$ depends on a small regularisation parameter $\epsilon>0$. 
In practice, one usually assumes $D=4-2\epsilon$ such that the ``physical limit'' is recovered when putting $\epsilon\rightarrow0$.

Despite the extreme complexity of Feynman integral calculations, the last decades have witnessed an impressive advancement in the identification of mathematical tools that can be put into action to perform these complicated computations.
One method that has proven itself to be spectacularly successful is the utilisation of differential equations satisfied by the Feynman integrals:
solving a system of differential equations for a given set of Feynman integrals, one can obtain the final result while circumventing the need to perform the original integrations \cite{Kotikov:1990kg,Bern:1993kr,Remiddi:1997ny,Gehrmann:1999as}.
These days, solving Feynman integrals via differential equations has become one of the standard ways to compute higher-order corrections for scattering processes.

Let us see how this method works in practice through a simple example.
Therefore, consider the following two Feynman integrals that are needed for a certain first-order correction in \ac{qed}:
\begin{align}
\begin{split}
    I_1&=\left(m^2\right)^{2-\frac{D}{2}}\int \frac{d^Dk}{i\pi^{\frac{D}{2}}}\frac{1}{\left[m^2-k^2\right]^2}\; ,\\
    I_2&=\left(m^2\right)^{3-\frac{D}{2}}\int \frac{d^Dk}{i\pi^{\frac{D}{2}}}\frac{1}{\left[m^2-k^2\right]^2\left[m^2-(k-p)^2\right]}\; .
\end{split}
\end{align}
In the above, $m$ denotes a real constant referring to a particle mass, whereas $p$ should be viewed as a variable referring to a particle momentum that may vary depending on the experimental setup.
In this sense, one may view $I_1$ and $I_2$ as functions depending on $p$.
Working in dimensional regularisation, we assume $D=4-2\epsilon$.
The two integrals $I_1$ and $I_2$ represent a particular choice of what is called a basis of master integrals.
More precisely, this means that all Feynman integrals that are relevant for computing the sought after perturbative correction can be reduced to $I_1$ and $I_2$.
It is an important fact that the choice of a basis of master integrals for a given perturbative correction is not unique.
As we will see below, for practical purposes, there are some choices of master integrals that are more appropriate than others.

Viewing $I_1$ and $I_2$ as functions of $x\defeq p^2/m^2$, we find the following differential equation for $\vec{I}=(I_1,I_2)^T$:
\begin{equation}
     \frac{d}{dx} \vec{I} =
\left( \begin{array}{cc}
 0 & 0 \\
 \frac{\hspace{6pt}\epsilon}{4 x} - \frac{\epsilon}{4\left(x-4\right)}\hspace{6pt} & \hspace{6pt}-\frac{1}{2x} - \frac{1+2\epsilon}{2\left(x-4\right)}\hspace{6pt}\\
 \end{array} \right)
 \vec{I}.
 \label{eq:nonCanonicalDE}
\end{equation}
Notice that all entries of the matrix on the right-hand side are rational functions of $x$.

Next, one tries to find what is called an $\epsilon$-\textit{decoupled basis} of master integrals \cite{Henn:2013pwa,Kotikov:2010gf}. 
Recall that we have some freedom in choosing a basis of master integrals for the perturbative correction at hand.
More precisely, it would be beneficial to bring the differential equation into a form, where the only explicit $\epsilon$-dependence is through a prefactor on the right-hand side.
To achieve this, we divide $I_1$ and $I_2$ by their maximal cuts \cite{Lee:2012te,Primo:2016ebd,Frellesvig:2017aai,Bosma:2017ens,Harley:2017qut,Lee:2017ftw,Chicherin:2018old}.
Changing our basis of master integrals from $I_1$ and $I_2$ to
\begin{equation}
    J_1=2\epsilon I_1,\hspace{8pt}J_2=2\epsilon\sqrt{-x(4-x)}I_2,
\end{equation}
the differential equation (\ref{eq:nonCanonicalDE}) becomes
\begin{equation}
     \frac{d}{dx} \vec{J} =\epsilon
\left( \begin{array}{cc}
 0 & 0 \\
 \hspace{6pt}-\frac{1}{\sqrt{-x(4-x)}}\hspace{6pt} & \hspace{6pt}-\frac{1}{x-4}\hspace{6pt} \\
 \end{array} \right)
 \vec{J}.
\end{equation}
The differential equation is now in $\epsilon$-\textit{decoupled form}.
Notice that, in order to obtain the $\epsilon$-decoupled form, we had to pay the price of introducing a square root in the matrix entries.
We may, however, change variables \cite{Barbieri:1972as} setting
\begin{equation}
    x=-\frac{(1-t)^2}{t}.
    \label{eq:varChangeSelfEnergy}
\end{equation}
This substitution turns the matrix entries into rational functions of the new variable $t$.
Indeed, we find
\begin{equation}
         \frac{d}{dt} \vec{J} =\epsilon
\left( \begin{array}{cc}
 0 & 0 \\
 \hspace{6pt}-\frac{1}{t}\hspace{6pt} & \hspace{6pt}\frac{1}{t}-\frac{2}{t+1}\hspace{6pt} \\
 \end{array} \right)
 \vec{J}.
\end{equation}
Having the differential equation in $\epsilon$-decoupled form and all matrix entries given as rational functions, it is straightforward to write down the final result for $J_1$ and $J_2$ in terms of \ac{mpls}.
\par\vspace{\baselineskip}
Though comparatively simple, the above considerations provide a typical example for the calculation of a given basis of master integrals.
While most steps can naturally be carried over to more complicated physical use cases, it turns out that one of the most demanding tasks is to find a change of variables like (\ref{eq:varChangeSelfEnergy}) that transforms the square roots appearing in the matrix entries into rational functions.
In the case of more ambitious perturbative corrections, this \textit{rationalisation problem} often marks an insurmountable difficulty for most practitioners.

\subsection{The problem of rationalising square roots}

Besides the success of momentum twistor variables \cite{Bourjaily:2018aeq,Hodges:2009hk,Gehrmann:2015bfy,Caron-Huot:2018dsv}, it was only recently that a systematic approach to the rationalisation problem was brought from mathematics to the physics community \cite{Besier:2018jen}. 
This approach relies on the fact that square roots can readily be associated with algebraic hypersurfaces.
For instance, a reasonable choice of a hypersurface associated with the above square root is the algebraic curve
\begin{equation}
    \mathcal{C}:y^2+x(4-x)=0.
\end{equation}
Notice that, if we are able to find a rational parametrisation of this curve, then we can use this parametrisation to turn the square root $\sqrt{-x(4-x)}$ into a rational function of $t$. 
Indeed, a possible parametrisation for $\mathcal{C}$ is
\begin{equation}
    x(t)=-\frac{(1-t)^2}{t},\hspace{12pt}y(t)=\frac{1-t^2}{t},
\end{equation}
corresponding to the change of variables given in (\ref{eq:varChangeSelfEnergy}).

In the above example, we are dealing with a plane conic curve. 
Thus, finding a rational parametrisation is an easy task.
Computing more sophisticated perturbative corrections, however, one is likely to encounter square roots for which the rationalisation problem is much more difficult.
For a long time, it was, for example, not clear to physicists how to find a change of variables that transforms the square root
\begin{equation}
    \sqrt{\frac{(x+y)(1+xy)}{x+y-4xy+x^2y+xy^2}}
    \label{eq:bhabhaRoot}
\end{equation}
into a rational function. 
This square root shows up in the context of second-order corrections to Bhabha scattering \cite{Henn:2013woa}, and it was recently proved to be non-rationalisable by showing that its associated hypersurface is birational to a K3 surface \cite{FvS19}.

Besides examples involving a K3, many perturbative corrections of the last years led to square roots associated with elliptic curves.
Such Feynman integrals can, in general, no longer be solved in terms of \ac{mpls} \cite{Laporta:2004rb}.
It was only recently that the notion of \ac{empls} was introduced \cite{Broedel:2017kkb,BrownLevin}, which finally enabled physicists to compute perturbative corrections whose analytic structure was previously out of reach.

\subsection{Motivation for a mathematical investigation of the Drell--Yan square root}

When trying to compute the master integrals for the mixed \ac{ew}-\ac{qcd} corrections to Drell--Yan scattering, one encounters the following square root \cite{Bonciani:2016ypc}:
\begin{equation}
    \sqrt{4xy^2(1+x)^2+\left(x(1+y)^2+y(1+x)^2\right)\cdot \left(x(1-y)^2+y(1-x)^2\right)}\, .
    \label{eq:DrellYanRoot}
\end{equation}
In an attempt to solve the integrals in terms of \ac{mpls}, one wants to know whether there exists a change of variables that turns (\ref{eq:DrellYanRoot}) into a rational function.
The answer to this question is an important physical motivation for this paper, and follows from Theorem \ref{t:Main1}.
\begin{corollary}
The square root (\ref{eq:DrellYanRoot}) cannot be rationalised by a rational variable change.
\end{corollary}
\begin{proof}
Suppose there would exist a rational variable change that rationalises (\ref{eq:DrellYanRoot}).
Then, it would be straightforward to write down a rational parametrisation for the surface $X_{DY}$ as defined in (\ref{eq:DY}).
In other words, $X_{DY}$ would be unirational.
However, Theorem \ref{t:Main1} tells us that $X_{DY}$ is birational to a K3 surface, i.e., its Kodaira dimension is 0.
Therefore, by the Enriques--Kodaira classification, $X_{DY}$ is not a rational surface; 
since unirationality and rationality are equivalent for surfaces over fields of characteristic 0 \cite[Remark V.6.2.1]{Hartshorne:1977}, $X_{DY}$ is not unirational.
\end{proof}

While this result provides very practical information, also other aspects of this paper might turn out to be useful for physicists.
For example, given that (\ref{eq:DrellYanRoot}) is not rationalisable by a rational variable change, one could hope that the geometry we encounter in the Drell--Yan case relates to a geometry in another physical process, e.g., to the K3 appearing in Bhabha scattering.
Finding such a correspondence could probably allow one to reuse some known techniques from the computation of the Bhabha correction and apply them in the context of the Drell--Yan correction.
A reasonable first attempt to formulate such a correspondence mathematically would be to ask for a birational map or, at least, an isogeny between the two K3 surfaces.
One way to answer this question is to compute the Picard lattice of the Drell--Yan K3 and compare it to the Picard lattice of the Bhabha K3. 
In this respect, Theorem~\ref{t:LambdaPic} tells us that this not the case, 
as shown by the following result (cf. Theorem~\ref{t:Main2}).
\begin{corollary}\label{c:NonIso}
The surface $X_{DY}$ is neither birationally equivalent nor isogenous to the surface $B$ arising from Bhabha scattering in~\cite{FvS19};
furthermore, it is not birationally equivalent to any of the deformations of $B$ considered in~\cite{FvS19}.
\end{corollary}
\begin{proof}
The surface $X_{DY}$ is birationally equivalent to its desingularisation $S:=S_{DY}$ (cf.~Definition~\ref{d:SDY}),
which is a K3 surface with geometric Picard number equal to $19$ (cf.~Proposition~\ref{p:DYK3} and Theorem~\ref{t:LambdaPic}).
The surface $B$ in~\cite{FvS19} is a K3 surface with geometric Picard number~$20$ (cf.~\cite[Main Theorem]{FvS19}).
For K3 surfaces, being birational and being isomorphic is equivalent.  
Furthermore, being isomorphic implies being isogenous.
This means that it is enough to show that $S$ and $B$ are not isogenous.
For two K3 surfaces over a field of characteristic $0$ to be isogenous, their geometric Picard numbers have to be equal (see~\cite[Proposition 13]{Sch13}).
Thus, there cannot exist an isogeny between $B$ and $S$.

To prove the second statement, 
it is enough to notice that the geometric Picard lattice of the generic deformation of $B$ considered in~\cite{FvS19} is not isometric to the geometric Picard lattice of $S$. 
\end{proof}

Finally, another aspect of our studies that might turn out to be useful for physicists is the investigation of elliptic fibrations.
On the one hand, we will see in Section \ref{sec:ComputingPicXViaEllipticFibrationsOnX} that elliptic fibrations of the Drell--Yan K3 can be used to compute its Picard lattice.
On the other hand, the aforementioned Bhabha correction was recently computed in terms of \ac{empls} \cite{tancredi:2019amp}, and, to the best of our knowledge, this computation was only possible because the Bhabha K3 has a certain elliptic fibration.
This suggests that a thorough study of elliptic fibrations could also give new insights to the Drell--Yan integrals and other perturbative corrections in \ac{qcd}---especially in view of the increasing number of physical computations that involve K3 surfaces \cite{Brown:2009ta,Brown:2010bw,Bourjaily:2018yfy,Bourjaily:2018ycu}.

\section{The Drell--Yan K3 surface and its Picard lattice} 
\label{sec:ComputingPicXViaDivisorsOnX}

Section~\ref{sec:PhysicalBackgroundAndMotivation} left us with some questions about the square root~\eqref{eq:DrellYanRoot}:
is it possible to find a change of variables turning it into a rational function?
Is it possible to find a change of variables such that the surface associated with it is birational to the K3 surface emerging from the Bhabha scattering?
In this section, we are going to show that both questions have a negative answer, hence proving the first two statements of Theorem~\ref{t:Main1} and Theorem~\ref{t:Main2}.

\subsection{The Drell--Yan K3 surface}
If $f(X,Y)$ is a polynomial of even degree $2d$,
then there is a natural way to associate a surface with the square root $\sqrt{f(X,Y)}$.
Let $\tilde{f}(x,y,z)$ be the homogenisation of $f$ via the substitution $X:=x/z$ and $Y:=y/z$.
If $u=\sqrt{f(X,Y)}$, then $uz^d=\sqrt{\tilde{f}(x,y,z)}$.
Substituting $uz^d$ with $w$ and squaring both sides,
we get the equation 
$$
w^2=\tilde{f}(x,y,z),
$$
which defines a surface in the weighted projective space $\PP(1,1,1,d)$ with coordinates $x,y,z$, and $w$, respectively.

Using the procedure above and rearranging the summands of the polynomial, 
one can easily see that~\eqref{eq:DrellYanRoot} is associated with the surface
$X_{DY}$ defined by
\begin{equation}
    w^2=4xy^2z(x-z)^2+(x+y)^2(xy+z^2)^2
\end{equation}
in the weighted projective space $\PP:=\PP (1,1,1,3)$ with coordinates $x,y,z,w$.
We define the map $\pi\colon X_{DY} \to \PP^2$  by $\pi \colon (x:y:z:w)\to (x:y:z)$.

\begin{lemma}
The surface $X_{DY}$ has five singular points, namely:
\begin{itemize}
    \item $P_1:=(1:1:-1:0)$, of type $A_1$;
    \item $P_2:=(0:0:1:0)$, \hspace{5pt} of type $A_2$;
    \item $P_3:=(1:-1:1:0)$, of type $A_3$;
    \item $P_4:=(1:0:0:0)$, \hspace{5pt} of type $A_4$;
    \item $P_5:=(0:1:0:0)$, \hspace{5pt} of type $A_4$.
\end{itemize}
\end{lemma}

\begin{proof}
The surface $X_{DY}$ is a double cover of $\PP^2$ branched above the plane curve $B\colon f=0$.
Therefore, the singularities of $X_{DY}$ come from the singularities of $B$,
which can easily be found by direct computations.
In order to find the type of singularity, 
it is enough to consider a double cover of the resolution of the singularities of $B$
(see, e.g., II, Sec. 8 and III, Sec. 7 of \cite{BHPV:04}). 
\end{proof}

\begin{definition}\label{d:SDY}
Let $S:=S_{DY}$ be the desingularisation of $X_{DY}$. Notice that $S$ is defined over $\QQ$. 
Let $\overline{\QQ}$ denote the algebraic closure of $\QQ$ inside $\CC$.
Then $\overline{S}$ denotes the change of base $S\times_{\QQ} \overline{\QQ}$ of $S$ to $\overline{\QQ}$.
\end{definition}
\begin{proposition}\label{p:DYK3}
$S$ is a K3 surface.
\end{proposition}
\begin{proof}
This follows from the theory of invariants of double covers with simple singularities as
described in V, Sec. 22 of \cite{BHPV:04}.
The map $\pi$ gives $X=X_{DY}$ the structure of a double cover of $\PP^2$, hence $H^1(X,\cO_X)=0$.
It ramifies above the plane curve defined by $f(x,y,z):=4xy^2z(x-z)^2+(x+y)^2(xy+z^2)^2=0$.
As $X_{DY}$ is normal, its canonical divisor $K_X$ is defined as the canonical divisor of its smooth locus.
As $f$ has degree six, $\pi$ is a double cover of $\PP^2$ ramified above a sextic, and hence $K_X=0$.

Because $S$ is the desingularisation of $X$, 
it is smooth.
We also have that $H^1(S,\cO_S)=0$ and
that the canonical divisor is unchanged by the resolutions,
since all the singular points are of A-type (DuVal singularities).
Therefore, $K_S=0$.
It follows that $S$ is a K3 surface.
\end{proof}

\begin{corollary}
The square root~\eqref{eq:DrellYanRoot} is not rationalisable by a rational change of variables.
\end{corollary}
\begin{proof}
Rationalising the square root~\eqref{eq:DrellYanRoot} by a rational change of variables is equivalent to finding a rational parametrisation of $X_{DY}$,
i.e., proving that $X_{DY}$ is a rational surface.
The desingularisation $S$ of $X_{DY}$ is a K3 surface, 
hence not a rational surface.
This implies that $X_{DY}$ is not a rational surface either. 
\end{proof}

\subsection{Computing the geometric Picard lattice}
In this subsection, we are going to compute the geometric Picard lattice of $S=S_{DY}$, that is, the Picard lattice of $\overline{S}$.
In doing so we follow the strategy explained in~\cite{Festi2} and we start by giving an upper bound of the Picard number of $\overline{S}$ using van Luijk's method with Kloosterman's refinement, cf.~\cite{vL07, Klo07}. 
\begin{proposition}\label{p:PicNumBound}
The surface $S$ has geometric Picard number $\rho (\overline{S})\leq 19$.
\end{proposition}
\begin{proof}
By Weil and Artin--Tate conjectures (both proven true for K3 surfaces over finite fields, cf.~\cite{Del72} for the proof of the Weil conjectures, and see for example~\cite{ASD73,Cha13} for the proof of the Tate conjecture for K3 surfaces over finite fields of characteristic $p\geq 5$),
we have that for a K3 surface $Y$ over a finite field with $q$ elements:
\begin{itemize}
    \item the Picard number equals the number of roots of the Weil polynomial of the surface which are equal to $\pm q$;
    \item  the discriminant of the Picard lattice is equal, up to squares, to the product $$\pm q\cdot \prod_{i=1+\rho(Y)}^{22}(1-\alpha_i/q)\; , $$
    where the $\alpha_i$'s denote the roots of the Weil polynomial. 
    They are ordered so that
    $\alpha_i=\pm q$ for $i=1,...,\rho(Y)$ and $\alpha_i\neq \pm q$   for $i=1+\rho(Y),...,22$ (cf.~\cite[Theorem 4.4.1]{Huy16}).
\end{itemize}
After checking that $31$ and $71$ are primes of good reduction for $S$, 
we use {\tt Magma} to compute the Weil polynomial of the reduction $S_q$ of $S$ over the finite field $\FF_q$, for $q=31,71$. 
This led us to the following results:
\begin{itemize}
    \item $\rho (\overline{S_{31}})=20$ and $\abs*{\disc \Pic \overline{S_{31}}}\equiv  3 \bmod (\QQ^*)^2$;
    \item $\rho (\overline{S_{71}})=20$ and $\abs*{\disc \Pic \overline{S_{71}}}\equiv  35 \bmod (\QQ^*)^2$.
\end{itemize}

Recall that the Picard lattice of a K3 surface over a number field injects into the Picard lattice of its reduction modulo a prime via a torsion-free-cokernel injection~(\cite[Proposition 13]{EJ11}).
Then, if we assume that $\rho (\overline{S})=20$,
it follows that $\disc \Pic \overline{S}=\disc  \Pic \overline{S_{31}}=\disc \Pic \overline{S_{71}}$ which is impossible, 
as the discriminants of $\Pic \overline{S_{31}}$ and $\Pic \overline{S_{71}}$ are not equivalent up to squares and therefore cannot be equal.
\end{proof}
\begin{remark}
As shown in Subsection~\ref{ssec:supersingular_reduction}, there are infinitely many primes of ordinary (i.e., not supersingular) reduction for $S$; let $q$ denote one of them.
As the Picard number of $\overline{S}$ is $19$ and $q$ is of regular reduction,  
the $20$ algebraic eigenvalues of the action of the Frobenius on the second cohomology group of $\overline{S_q}$ will be equal to $\pm q$.
If one chooses $q$ so that the reduction $\textrm{mod} q$ of the divisors in $\Sigma$ (see below) is defined directly on $\FF_q$ and no extension is needed (this just means that $\FF_q$ contains a square root of $5$), then nineteen of the $20$ algebraic eigenvalues will be equal to $q$, with the last one left uncertain. 
{\it A priori} it is not possible to determine its sign and, 
after a short search, the primes $q=31, 71$ turned out to be the smallest ones for which also the twentieth eigenvalue equals $q$.
Nevertheless, {\it a posteriori}, one can use Theorem~\ref{t:SIStructure} to determine the sign of the twentieth eigenvalue: it equals the Kronecker symbol $\left(\frac{10}{q}\right)$.
\end{remark}
We show that the geometric Picard number $\rho (\overline{S})$ is exactly $19$ by considering the sublattice generated by the following divisors.
As $S$ is the resolution of $X_{DY}$, on $S$ we have the exceptional divisors lying above the singular points of $X_{DY}$,
namely:
\begin{itemize}
    \item $E_{1,1}$ above the point $P_1$;
    \item $E_{2,-1}$ and $E_{2,1}$ above $P_2$;
    \item $E_{3,-1}, E_{3,0},$ and $E_{3,-1}$ above $P_3$;
    \item $E_{i,-2}, E_{i,-1}, E_{i,1}$, and $E_{i,2}$ above $P_i$, for $i=4,5$.
\end{itemize}
Furthermore, consider  the following nine divisors of $X_{DY}$.
\begin{align*}
    L'_1\colon & x=0,\hspace{4pt}w-yz^2=0\, ;\\
    L'_2\colon & x-z=0,\hspace{4pt} w-z(y+z)^2=0\, ;\\
    L'_3\colon & y+\frac{3-\sqrt{5}}{2}z=0,\hspace{4pt}w+\frac{3-\sqrt{5}}{2}z(x^2-xz+z^2)=0\, ;\\
    L'_4\colon & y+\frac{3+\sqrt{5}}{2}z=0,\hspace{4pt}w+\frac{3+\sqrt{5}}{2}z(x^2-xz+z^2)=0\, ;\\
    L'_5\colon & y=0,\hspace{4pt} w-xz^2=0\, ;\\
    L'_6\colon & y+z=0, \hspace{4pt} w-z(x-z)(x+z)=0\, ;\\
    L'_7\colon & z=0,\hspace{4pt} w-xy(x+y)=0\, .
\end{align*}
\begin{align*}
        C'_1\colon & x^2+\frac{-1+\sqrt{5}}{2}(xy+xz)+yz=0, \\ 
    & xy^2 + \frac{5+\sqrt{5}}{2}xyz +\frac{5+3\sqrt{5}}{2}y^2z +\frac{3+\sqrt{5}}{2}xz^2 +\frac{5+3\sqrt{5}}{2}yz^2 +\frac{3+\sqrt{5}}{2} w=0\\
    C'_2\colon & x^2+\frac{-1-\sqrt{5}}{2}(xy+xz)+yz=0,\\ 
    &xy^2 + \frac{5-\sqrt{5}}{2}xyz +\frac{5-3\sqrt{5}}{2}y^2z +\frac{3-\sqrt{5}}{2}xz^2 +\frac{5-3\sqrt{5}}{2}yz^2 +\frac{3-\sqrt{5}}{2} w=0
\end{align*}
For $i=1,...,7$ and $j=1,2$, we define $L_i$ and $C_j$ to be the strict transform of $L'_i$ and $C'_j$, respectively, on $S$.
Finally, let $H'$ denote the hyperplane section on $X_{DY}$;
we define $H$ to be the pullback of $H'$ on $S$.
Let $\Sigma$ be the set of the 24 divisors of $S$ defined so far,
and let $\Lambda\subseteq \Pic S$ be the sublattice of $\Pic S$ generated by the classes of the elements in $\Sigma$.
Recall that the \textit{discriminant group} of a lattice $L$ is defined to be the quotient $A_L:=\Hom (L, \ZZ) / L$.

\begin{proposition}\label{p:PicSub}
The sublattice $\Lambda$ has rank $19$, discriminant $2^3\cdot 3$ and discriminant group isomorphic to $\ZZ/2\ZZ \times \ZZ/2\ZZ\times \ZZ/6\ZZ$.
\end{proposition}
\begin{proof}
The intersection matrix of these divisors has been computed using a built-in {\tt Magma} function to determine the intersection numbers between the strict transforms of surface divisors and the exceptional divisors and a custom function for the local intersection numbers between the strict transforms over the singular points.
This function can be found in the accompanying file~\cite{BFHN19}.
See also Remark~\ref{r:computations}.
\end{proof}

\begin{remark}
Proposition~\ref{p:PicSub} can be proven also in a different way, 
less explicit but also involving fewer coding skills.
Indeed, notice the following properties.
\begin{itemize}
    \item $H^2=2$; for every $i,j, \; H.E_{i,j}=0$; for every $i,\; H.L_i=1$;
    for every $j,\; H.C_j=2$.
    \item For every $i,j,\; C_j^2=L_i^2=-2$. 
    The intersection numbers $L_i.L_j,\, L_i.C_j,\, C_i.C_j$ can be explicitly computed as we have the explicit defining equations.
    \item The intersection numbers of the exceptional divisors are completely determined by the type of singularity.
    \item A few intersection numbers between the exceptional divisors and the divisors $L_i, C_j$ can be determined by an \textit{ad hoc} labelling of the exceptional divisors. (See Example~\ref{ex:labeling} for an instance of this labelling.) 
\end{itemize}
These remarks still leave some intersection numbers undetermined. 
These undetermined numbers can either be $1$ or $0$.
Using a computer, one can go through all the combinations and find that only one  satisfies the condition $\rk \Lambda \leq 20$.
This combination returns the quantities in the statement.
These computations can be found in the accompanying  file.
\end{remark}
\begin{example}\label{ex:labeling}
The line $\ell_2:=\{x=0\}\subset \PP^2$ passes through the point $(0:0:1)$, this means that one of the exceptional divisors $E_{2,1}$ and $E_{2,1}$ intersects $L_2$: we denote by $E_{2,-1}$ the one intersecting it, hence $E_{2,-1}.L_2=1$.
As $\{ x = 0\}$ is not in the tangent cone of the branch curve at $(0:0:1)$, it follows that $E_{2,1}.L_2=0$.
\end{example}

\begin{theorem}\label{t:LambdaPic}
$\Pic \overline{S}=\Lambda$.
\end{theorem}
\begin{proof}
In this proof, we denote $\Pic \overline{S}$ by simply $P$.
From Propositions~\ref{p:PicNumBound} and~\ref{p:PicSub} it immediately follows that $P$ has rank $19$ and hence $\Lambda$ is a finite-index sublattice of $P$.
As the discriminant of $\Lambda$ is $24=2^3\cdot 3$,
the index $[P : \Lambda]$ is either $1$ or $2$.

For a contradiction, assume $[P : \Lambda]=2$ and
let $\iota \colon \Lambda \hookrightarrow P$ be the inclusion map.
Then the induced map $\iota_2 \colon \Lambda/2\Lambda \to P/2P$ has exactly one non-zero element in its kernel $\ker \iota_2=\frac{\Lambda \cap 2P}{2\Lambda}$.
Let $\Lambda_2$ be the set 
$$
\{ [x]\in \Lambda/2\Lambda \; :\; \forall \, [y]\in \Lambda/2\Lambda ,\; x.y\equiv 0 \bmod 2 \;\;\text{ and }\;\; x^2\equiv 0 \bmod 8\}.
$$
Notice that $\Lambda_2$ contains $\ker \iota_2$ and it can be explicitly computed as it only depends on $\Lambda$, which we know.
Then one can see that $\Lambda_2$ contains two non-zero elements, say $v_1, v_2$. 
As we assumed $[P:\Lambda]=2$, only one between $v_1$ and $v_2$ is in $\ker \iota_2$.

As $P$ is defined over $\QQ (\sqrt{5})$, the Galois group 
$$
G:=\Gal (\QQ (\sqrt{5})/\QQ)=\langle \sigma \rangle \cong \ZZ/2\ZZ
$$ 
acts on $P$
and the kernel $\ker \iota_2$ is invariant under this action.
By explicit computations, one can show that $v_1$ and $v_2$ are conjugated under the action of $G$,
hence if one is in $\ker \iota_2$ also the other is, getting a contradiction.

We can then conclude that $[P:\Lambda]=1$, that is, $\Lambda=\Pic \overline{S}$.
\end{proof}

\subsection{An application: computation of the Brauer group} \label{ssec:ComputationOfTheBrauerGroupOfX}

In this subsection, we obtain information about the algebraic part of the Brauer group of $S=S_{DY}$ using the Galois module structure of $\Pic \Sbar$.
Let $\Br S$ denote the Brauer group of $S$ and recall the filtration
$$
\Br_0 S \subseteq \Br_1 S \subseteq \Br S,
$$
where $\Br_0:= \im (\Br (\QQ) \to \Br (S))$ 
and $\Br_1:= \ker (\Br {S} \to (\Br \Sbar)^G)$, the algebraic part of~$\Br S$.

Earlier in Section~\ref{sec:ComputingPicXViaDivisorsOnX} we have explicitly given a set of divisors $\Sigma$ generating the whole geometric Picard lattice of $S$.
In particular it turns out that $\Pic \Sbar$ is defined over the field $\QQ (\sqrt{5} )$, a quadratic extension of $\QQ$.
It follows that there is an action of the Galois group
$ G :=\Gal (\QQ (\sqrt{5} ) / \QQ ) \cong \ZZ/2\ZZ$ over $\Pic \Sbar$.
As $\Sigma$ is invariant under the action of $G$, 
after choosing a basis of $\Pic \Sbar$, 
it becomes straightforward to explicitly describe the action using $19\times 19$ matrices.
Using this description it is then possible to compute the cohomology groups
$H^i(G, \Pic \Sbar)$ for every $i$.
\begin{proposition}\label{p:CohomologyGrps}
The following isomorphisms hold:
\begin{enumerate}
    \item $H^0(G, \Pic \Sbar)\cong \ZZ^{18}$;
    \item $H^1 (G, \Pic \Sbar)\cong 0$;
    \item $H^2 (G, \Pic \Sbar)\cong (\ZZ/2\ZZ)^{17}$.
\end{enumerate}
\end{proposition}
\begin{proof}
Explicit  computations, see the attached file~\cite{BFHN19}.
\end{proof}
\begin{corollary}
The quotient $\Br_1 S/\Br_0 S$ is trivial.
\end{corollary}
\begin{proof}
As $S$ is defined over $\QQ$, a global field,
from the Hochschild--Serre spectral sequence we have an isomorphism 
$$
\Br_1 S/\Br_0 S \cong H^1(G, \Pic \Sbar).
$$
(See, for example, \cite[Corollary 6.7.8 and Remark 6.7.10]{Poo17}.)
From Proposition~\ref{p:CohomologyGrps} it follows that $H^1(G, \Pic \Sbar)=0$,
proving the statement.
\end{proof}

\section{Elliptic fibrations on the surface} \label{sec:FullClassificationOfEllipticFibrationsOnX}

In this section, we explicitly describe some elliptic fibrations of the surface $S$ which are used in Section~\ref{sec:ComputingPicXViaEllipticFibrationsOnX} to re-obtain the full geometric Picard lattice of $S$ in a different way.

After briefly recalling some basic notions concerning elliptic fibrations in Subsection~\ref{ssec:BasicsEllFibr},
we present a general method to find elliptic fibrations on K3 surfaces with many $-2$-curves (i.e., curves whose self-intersection is $-2$) 
in Subsection~\ref{ssec:Stats}.
We apply this method to the surface $S$, deriving some statistics about elliptic fibrations on $S$.
In the last three subsections, we give the explicit description of three elliptic fibrations of $S$.
The second fibration given in Subsection~\ref{ssec:Second} is the one used in
Section~\ref{sec:ComputingPicXViaEllipticFibrationsOnX};
the third fibration, given in Subsection~\ref{ssec:Third} is the one used to give an explicit description of the Shioda--Inose structure of $S$ (see Section~\ref{sec:ComputationOfTheShiodaInoseStructureOfX}).
The first fibration is the simplest to compute and is the one that we originally used for the computations in Section~\ref{sec:ComputingPicXViaEllipticFibrationsOnX}. We include it here as another useful example although
the second fibration is more straightforward to use for our purposes.

\subsection{Background on elliptic fibrations}\label{ssec:BasicsEllFibr}

The {\it elliptic fibrations} of $S$ are of interest for several reasons. 
We already touched upon their role in physics in Section \ref{sec:PhysicalBackgroundAndMotivation}.
Apart from that, they can also be used to provide an alternative proof of the full structure of $\Pic \overline{S}$, as will be demonstrated in Section~\ref{sec:ComputingPicXViaEllipticFibrationsOnX}, and to compute other important arithmetic structures (Section~\ref{sec:ComputationOfTheShiodaInoseStructureOfX}).

\begin{definition}
An \textit{elliptic fibration} is a morphism $\phi$ of $S$ onto the projective line $\PP^1$ 
(over~$\Qbar$) with general fibre a non-singular genus $1$ curve,
which has a section. A section of $\phi$ is a curve in $S$ which maps isomorphically down to
$\PP^1$ under $\phi$. 
\end{definition}
Equivalently, a section of $\phi$ is an irreducible curve in $S$ whose intersection number
with a fibre of $\phi$ (all of which are rationally equivalent) is $1$. 
Two fibrations $\phi$ and
$\phi_1$ are said to be equivalent if there is an automorphism of $\PP^1$, $\alpha$,
such that $\phi_1 = \alpha \circ \phi$. 
Any such $\alpha$ is given by the standard
$ x \mapsto (ax+b)/(cx+d)$ action of an invertible $2$-by-$2$ matrix
$\bigl(\begin{smallmatrix}
a&b \\ c&d
\end{smallmatrix} \bigr)$
which is determined by $\alpha$ up to multiplication by a non-zero scalar.
If only the morphism $\phi$ is given, without any section (and possibly having no section), then we call $\phi$
a \textit{genus $1$ fibration}.

\begin{theorem}[\cite{PSS:71}]\label{t:EllFibGen71}
Let $Y$ be a K3 surface.
Then, genus $1$ fibrations of $Y$ (up to equivalence) are in 1-1 correspondence
with divisor classes $E$ in $\Pic \overline{S}$ which satisfy
\begin{enumerate}
\item $E.E = 0$ ($E$ has self-intersection $0$),
\item $E$ is primitive (i.e., the class $E$ is not divisible by any $n \geq 2$ in 
  $\Pic \overline{S}$),
\item $E$ lies in the {\it nef} cone (i.e., $E$ has non-negative intersection number
with all classes in $\Pic \overline{S}$ that represent effective divisors). 
\end{enumerate}
Under this correspondence, a genus $1$ fibration $\phi$ is associated with the class of its fibres; 
conversely, a divisor class $E$ satisfying the conditions 1,2,3 corresponds to the fibration map class
given by the map to ${\PP}^1$ associated with the Riemann--Roch space of its global sections.
\end{theorem}
\begin{proof}
This is \cite[Theorem 1]{PSS:71}.
\end{proof}

For general results relating
maps to projective space, invertible sheaves and divisor classes up to rational equivalence see
\cite[Section II.7]{Hartshorne:1977}, 
for specific results about general linear systems of K3 surfaces
see \cite{SaintDonat:1974}.
\medskip

If $\phi$ is a genus $1$ fibration,
the condition for the existence of a section is that there is another class $D$ such that
the intersection number $E. D$ is $1$. 
From this, it can then be seen
that pairs of (fibration, section) correspond to $2$-dimensional {\it hyperbolic} direct
summands of the $\Pic \overline{S}$ lattice. 
For a fixed fibration, any two distinct
sections are mapped to each other under an isomorphism of $S$ that preserves the fibration,
viz. a translation map on the generic fibre extended to an automorphism of $S$. Since,
ultimately, we are only interested in elliptic fibrations up to $\Aut \overline{S}$, 
we will not worry
too much about differentiating between different sections of an elliptic fibration. 
The method we use (which dates back in the literature to at least  \cite{ShiodaInose:1977}) for explicitly constructing fibrations gives genus $1$ fibrations, 
although it also finds explicit sections in the majority of cases.

\begin{definition}
The {\it generic fibre} of a fibration $\phi:S\rightarrow\PP^1$ over a field $k$, is the
genus $1$ curve $S_t/k(t)$ defined as the pullback of $\phi$ under the generic point
inclusion $\Spec(k(t)) \hookrightarrow \PP^1$. That is, $S_t$ is the
fibre product $S \times_{\PP^1} \Spec(k(t))$.
\end{definition}
 If $s$ is a section of $\phi$, then 
the analogous pullback gives a point $s_t\in S_t(k(t))$ which we can take as the
$O$-point for an elliptic curve structure on $S_t$. When we talk about a fibration
with a section, it is to be assumed that we are considering $S_t$ as an elliptic 
curve with $s_t$ as $O$. 
The surface $S$ can be recovered from the generic fibre $S_t$, being
isomorphic to the {\it Minimal (Curve) Model} over $\PP^1$ of $S_t/k(t)$,
which is characterised as a non-singular flat, proper scheme over (the Dedekind scheme)
$\PP^1$, whose generic
fibre is isomorphic to $S_t/k(t)$ and which has no $-1$-curve as a component of any fibre.
For more information on minimal models of curves, see~\cite[Chapter XIII]{ArithGeom:1986}.

As we shall see shortly, $S$ has infinitely many elliptic fibrations up to equivalence.
For any K3 surface, however, there are only finitely many classes of elliptic fibrations
up to the action of $\Aut \overline{S}$. There are some cases, using lattice computations on
the full $\Pic \overline{S}$, where a complete set of classes have been
calculated (e.g. \cite{Nis96}, \cite{Nis97}, \cite{Ogu89} or more recent papers \cite{BKW:2013} and \cite{BGHLMSW:2015}),
but we have not attempted to carry out such a computation here.
Instead, we have used a method, described in the next subsection,
for computing elliptic fibrations, which is independent of the knowledge of the full
Picard group, and produces a large number of inequivalent fibrations when applied to
$S$ and the set of $-2$-curves from the last section. An interesting subset of these 
fibrations, which we have used for various computations, will be given explicitly 
in the following subsections.

\subsection{General method and results for the Drell--Yan surface}\label{ssec:Stats}

Consider a collection of $-2$-curves, $\{C_i\}_{i \in I}$, on a smooth projective surface. Let
$$ D = \sum_{i \in I} m_iC_i $$
be an effective divisor supported on a subset of the $\{C_i\}$.
We refer to $D$ as a
{\it Kodaira fibre} when the configuration of the curves occurring in $D$ is that
of one of Kodaira's singular fibres for an elliptic fibration and they occur in $D$
with the correct multiplicities for that type of fibre (see, e.g., \cite[Section V.7]{BHPV:04}). For example, $C_1+C_2+C_3$ is a Kodaira fibre of type $I_3$ if 
$C_i$ and $C_j$ meet transversally in a single point, for $1 \le i < j \le 3$, and
the three intersection points are distinct.

The terminology that we use to refer to the type of a Kodaira fibre is primarily the 
Dynkin style (apart from using $I_{n+1}$ rather than $\widetilde{A_n}$)
$I_n$, $\widetilde{D_n}$ and $\widetilde{E_n}$. Types $II$, $III$ and $IV$ do not occur in this paper.

\begin{lemma}\label{l:EllFibGen}
\begin{itemize}
    \item[(a)] A Kodaira fibre $D$ on a K3 surface $X$ is always a singular fibre for an elliptic fibration of the surface. The associated linear system $|D|$ is base-point
    free of dimension $1$ and the associated map $\phi_D : X \rightarrow \PP^1$
    gives the elliptic fibration.
    \item[(b)] Distinct Kodaira fibres $D_1,\ldots,D_n$ lead to the same elliptic fibration up to equivalence
    if and only if they are pairwise disjoint (i.e., non-intersecting), in which case they give different singular 
    fibres of that fibration. 
\end{itemize}
\end{lemma}
\begin{proof}
(a) Essentially, this is just \cite[Lemma 1.1]{ShiodaInose:1977}, following easily from Theorem \ref{t:EllFibGen71}, cf. \cite[Theorem $\theta$, p. 13]{BKW:2013}.

(b) Consider the fibration determined by $D_1$. Since
the curves in the other $D_i$ do not intersect $D_1$, they cannot cover the base $\PP^1$ of the fibration.
Thus they lie in fibres distinct from $D_1$. Each $D_i$ is connected, so lies in a single fibre. By
Lemma 1.2 {\it loc. cit} (essentially Zariski's lemma for fibres), and the facts (from the definition of a Kodaira fibre) that $D_i.D_i=0$ and $D_i$ has an irreducible component of multiplicity $1$, so can't be a multiple of a fibre, each $D_i$ is the entire fibre. The fibres are distinct
because the $D_i$ are pairwise disjoint. 
\end{proof}
\begin{remark} 
Equivalent elliptic fibrations have the same set of fibres and are completely determined by any one of those fibres.
The fibres are all rationally equivalent and are the effective divisors of a single linear system $|D|$. The fibration
is the one corresponding to that linear system up to equivalence: $\phi$ is the map to ${\PP}^1$ associated to $|D|$.
\end{remark}

\noindent
{\bf The method.}
Let $Y$ be a K3 surface on which many $-2$-curves are known.
The method consists in constructing elliptic fibrations by searching for Kodaira fibres supported on these curves. 
Additionally, we hope to find explicit sections for a fibration from
amongst the same set of curves.
\begin{enumerate}
    \item Let $\mathcal{S}$ be a finite set of known $-2$-curves on $Y$.
    \item Compute the intersection matrix $M$ of the curves in $\mathcal{S}$.
    \item Find all the Kodaira fibres supported by curves in $\mathcal{S}$ (purely combinatorial).
    \item Compute the elliptic fibration for interesting Kodaira fibres $D$, i.e., compute the Riemann--Roch space for $D$.
    \item For a Kodaira fibre $D$, find a section in $\mathcal{S}$ of the elliptic fibration determined by $D$, i.e., 
    find a $-2$-curve $C$ supported in $\mathcal{S}$ such that $D.C=1$.
\end{enumerate}

\begin{remark}\label{r:computations}
For the computations on the surface $S$ in this paper,
the set $\mathcal{S}$ is defined below.
The matrix $M$ has been computed  using a {\tt Magma} functionality to determine the
intersection numbers between the strict transforms of divisors on the singular surface model and
the exceptional divisors and also the local
intersection numbers between the strict transforms over the singular points.
We automated Step 3 of this method with a  function that takes $M$ as argument. 
Step 4 was achieved using a slightly adapted version of {\tt Magma}'s
standard Riemann--Roch functionality.
The computations also made use of the {\tt Magma}
function to impose additional Riemann--Roch conditions at singular points on the
surface in order to handle the exceptional divisors correctly.
See the file attached ~\cite{BFHN19}.
\end{remark}

We have the $-2$-curves $L_i$, $1 \le i \le 7$, and $C_1, C_2$ along with the fourteen
$E_{i,j}$ exceptional divisors from the last section. 
Also, we have the
transforms of the $L_i$ and $C_j$ under the $w \mapsto -w$ automorphism of $S$.
We denote these transforms by $\tilde{L_i}$ and $\tilde{C_j}$. Finally, we 
consider one final curve $C_3$ and its transform $\tilde{C_3}$, where $C_3$ was
also found as in Section~\ref{sec:ComputingPicXViaDivisorsOnX} and is the strict transform on $S$ of
$$ C'_3\colon x^2+yz=0,\hspace{4pt} xy^2-y^2z-xz^2-yz^2-w=0. $$
Then we define
$$ \mathcal{S} := \{L_i, \tilde{L_i} : 1 \le i \le 7\} \cup  
 \{C_i, \tilde{C_i} : 1 \le i \le 3\} \cup \{E_{i,j}\}, $$
 a set of thirty-four $-2$-curves on $\overline{S}$.
In summary, we found the following results for $S$.
 
 \begin{theorem}
 Let $\Gamma\subset \Aut \overline{S}$ be the subgroup of automorphisms generated by  $w \mapsto -w$ and  $\sqrt{5} \mapsto -\sqrt{5}$.
 Then on $\overline{S}$ there are
 \begin{itemize}
     \item $105,856$ Kodaira fibres supported on $\mathcal{S}$, leading to
     \item $104,600$ different genus $1$ fibrations, $86,416$ having
     a section in $\mathcal{S}$;
     \item $29,111$  fibrations inequivalent up to action of $\Gamma$, of which $27,807$ have a section and
     $24,270$ have a section in $\mathcal{S}$.
 \end{itemize}
There are Kodaira fibres of types $I_n$, $2 \le n \le 14$ and $n=16$;
$\widetilde{D}_n$, $n \le 4 \le 10$ and $n \in\{12,14,16\}$; and $\widetilde{E}_6, \widetilde{E}_7, \widetilde{E}_8$.
\end{theorem}
\begin{proof}
By explicit computations. See the file attached~\cite{BFHN19}.
We note that our program for finding configurations of $(-2)$ curves in $\mathcal{S}$ that give Kodaira fibres would label any type $III$ fibres as $I_2$ and any type $IV$ fibres as $I_3$. These pairs of type configurations are indistinguishable purely from intersection numbers. However, as explained in the attached file, we
checked that no three curves in $\mathcal{S}$ meet in a single point and
no pair intersect tangentially. Thus, no type $III$ or $IV$ fibres can occur here.
\end{proof}

\begin{remark}
The number of distinct bad fibres entirely
supported on $\mathcal{S}$ in the various cases ranges from 1 to 5.
The four-fibre cases have $I_{10},I_2,I_2,I_2$ or 
$I_8,\widetilde{D}_5,I_2,I_2$ type $\mathcal{S}$-supported fibres and the single five-fibre case has $I_6,I_6,I_6,I_2,I_2$ type $\mathcal{S}$-supported fibres.
We have not computed the full sets of bad fibres in every case or 
attempted to determine how many classes of fibrations the $104,600$
give under the full action of $\Aut \overline{S}$ (and $\Gal(\overline{\QQ}/\QQ)$).
They surely are many fewer in number than $29,111$.
\end{remark}

The following lemma gives the $\mathcal{S}$-Kodaira fibre data for three particular fibrations to be
used in the next section of the paper. Explicit forms are given in the next three subsections.
We note now that the first two have infinite Mordell--Weil groups, so by a result
of Nikulin (cf.~\cite[Theorem 9]{Nikulin:2013}), {\it there are infinitely many
inequivalent elliptic fibrations}!

\begin{lemma}\label{l:FibData}
(a) There is an elliptic fibration of $S$ over $\QQ$ with three bad fibres consisting entirely
of curves in $\mathcal{S}$:
\begin{itemize}
    \item[(i)] an $I_6$ fibre, $L_1+E_{2,1}+E_{2,-1}+\tilde{L}_1+E_{5,-1}+E_{5,1}$;
    \item[(ii)] an $I_6$ fibre, $\tilde{L}_7+L_7+E_{4,2}+E_{4,1}+E_{4,-1}+E_{4,-2}$;
    \item[(iii)] a $\widetilde{D_4}$ ($I_0^*$) fibre, $L_2+\tilde{L}_2+E_{3,-1}+E_{3,1}+2E_{3,0}$;
\end{itemize}
where the sums for the $I_6$ fibres give the components in cyclic order.

The $L_i$ and $\tilde{L}_i$ not occurring in one of these fibres along with the
$C_i$, $\tilde{C}_i$ and $E_{5,-2}, E_{5,2}$ all give sections of the fibration.
\smallskip

\noindent (b) There is an elliptic fibration of $S$ over $\QQ$ with four bad fibres consisting entirely of curves
in $\mathcal{S}$:
\begin{itemize}
    \item[(i)] an $I_{10}$ fibre, $L_6+L_1+E_{5,1}+E_{5,-1}+\tilde{L}_1+\tilde{L}_6+E_{4,2}+E_{4,1}+E_{4,-1}+E_{4,-2}$;
    \item[(ii)] an $I_2$ fibre, $C_1+\tilde{C}_1$;
    \item[(iii)] an $I_2$ fibre, $C_2+\tilde{C}_2$;
    \item[(iv)] an $I_2$ fibre, $C_3+\tilde{C}_3$;
\end{itemize}
where the sum for the $I_{10}$ fibre give the components in cyclic order.

The curves $L_5, \tilde{L}_5, L_7, \tilde{L}_7, E_{2,1}, E_{2,-1}, E_{3,1}, E_{3,-1}, E_{5,2}, 
E_{5,-2}$ all give sections of the
fibration. 
All other curves in $\mathcal{S}$ apart from these and the ones occurring 
in the above fibres
give $2$-sections or lie in other bad fibres.
\smallskip

\noindent (c) There is a genus $1$ fibration of $S$ over $\QQ(\sqrt{5})$ with two bad fibres consisting entirely of
curves in $\mathcal{S}$:
\begin{itemize}
    \item[(i)] an $\widetilde{E}_8$
    ( $II^*$ ) fibre,
    $6E_{4,-2}+5E_{4,-1}+4E_{4,1}+4L_6+3\tilde{L}_3+3L_5+2E_{1,1}+2E_{2,1}+E_{2,-1}$;
    \item[(ii)] another $\widetilde{E}_8$ ($II^*$) fibre,
    $6\tilde{L}_2+5E_{5,-2}+4E_{5,-1}+4E_{3,0}+3L_4+3E_{5,1}+2E_{3,-1}+2E_{5,2}+L_7$.
\end{itemize}
The fibration has no sections. However, there are a number of $2$-sections provided by curves in
$\mathcal{S}$: in particular, $\tilde{L}_6$.
\end{lemma}
\begin{proof}
This follows from the computations described above. The fact that the fibration in (c) has no section
comes from the intersection matrix $M$, which shows that the intersection number of each curve in $\mathcal{S}$
with either fibre is even. Note that $\mathcal{S}$ generates $\Pic \overline{S}$ (cf. Theorem~\ref{t:LambdaPic}).
\end{proof}
\bigskip

\subsection{First elliptic fibration}\label{ssec:First}

\begin{proposition}\label{p:EllFibOne}
(a) The generic fibre $S_t/\QQ(t)$ of the elliptic fibration of Lemma \ref{l:FibData} (a)
has a Weierstrass equation
$$ E_1: y^2=x^3+(t-1)^2(t^2+6t+1)x^2-16t^3(t-1)^2x. $$
\noindent (b) The full set of bad fibres for this fibration is given by the following fibres.
\begin{itemize}
    \item The $I_6,I_6,I_0^*$ fibres (i),(ii) and (iii) of Lemma~\ref{l:FibData}.\\ 
    They lie  over $t=\infty,0,1$, respectively.
    \item An $I_2$ fibre over $t=-1$ with components $E_{1,1}$ and the strict transform on $S$ of 
    the pullback on $X$ of the plane curve $x+z=0$.
    \item Four $I_1$ fibres over the points satisfying $t^4+8t^3-2t^2+8t+1=0$.
\end{itemize}
\smallskip

\noindent (c) The group of points on $E_1(\QQ(\sqrt{5})(t))$
generated by the $\mathcal{S}$-sections listed in Lemma~\ref{l:FibData} is isomorphic to $$\ZZ/2\ZZ\oplus\ZZ\oplus\ZZ\, ,$$ 
where the first summand is generated by the $2$-torsion point $(0,0)$ and the two free ones are generated
by the points
$$ P_1=(4t(t-1),-4t(t+1)(t-1)^2) \qquad\hbox{and}\qquad P_2=(4t^3(t-1),-4\sqrt{5}t^3(t+1)(t-1)^2). $$
\end{proposition}
\begin{proof}
(a) Computing the Riemann--Roch space for Kodaira fibre (i) of Lemma~\ref{l:FibData} (a), using {\tt Magma}, 
gives the fibration map
$$ S \longrightarrow {\PP}^1 \qquad [x:y:z:w] \mapsto [z:x]\; . $$
Letting $t=z/x$, we computed a singular plane model of $S_t$ in weighted
projective space $\PP(1,2,1)$ over $\QQ(t)$ with variables $a,b,c$ via
the substitution $x=(1/t)c, y=a, z=c, w=bc$ of the form $b^2 =f(a,c)$ for
a homogeneous quartic $f$.
We also computed the $\QQ(t)$-rational (non-singular) point corresponding to the $L_5$ section on this model of $S_t$.
Then a curve Riemann--Roch computation using {\tt Magma}'s function field machinery gives a Weierstrass model
for $S_t$, which is easily simplified to the $E_1$ model given. The explicit isomorphism  from the singular plane
model to $E_1$ is messy, and we do not write it down here, but it can
be derived from the computations in the attached  file~\cite{BFHN19}, which contains full details of all of the above.
\medskip

(b) This follows easily from applying Tate's algorithm to the $E_1$ model.
\medskip

(c) The points in $E_1(\QQ(\sqrt{5})(t))$ corresponding to the sections were computed firstly on
the plane model of $S_t$ via the variable substitution given in (a), and then on $E_1$ using
the explicit map from $S_t$ to it. The result is then an easy lattice computation given the canonical height pairings between
the points, which were computed for simplicity with the standard {\tt Magma} intrinsic {\tt HeightPairingMatrix}.
Note that we could have also just used the intersection pairings from the matrix $M$, from which canonical
heights are easily deduced since $S$ is the minimal model of $S_t$.
More computational details are in the attached file~\cite{BFHN19}
\end{proof}
\begin{remark}
There is a $t\mapsto 1/t$ symmetry and setting $s=t+(1/t)-2$, we see that $E_1$
is the base change under $\QQ(s)\hookrightarrow\QQ(t)$ of $Y^2=X^3+s(s+8)X^2-16sX$,
which is the generic fibre of a rational elliptic surface.
\end{remark} 
\bigskip

\subsection{Second elliptic fibration}\label{ssec:Second}

\begin{proposition}\label{p:EllFibTwo}
(a) The generic fibre $S_t/\QQ(t)$ of the elliptic fibration of Lemma \ref{l:FibData} (b)
has a Weierstrass equation
$$ E_2: y^2=x^3-(3t^4+8t^3-2t^2-1)x^2+16t^5(t^2+t-1)x. $$
\noindent (b) The full set of bad fibres for this fibration is given by the following fibres.
\begin{itemize}
    \item The $I_{10},I_2,I_2,I_2$ fibres (i),(ii),(iii) and (iv) of Lemma~\ref{l:FibData}.\\
    They lie over $t=0,
    -(\sqrt{5}+1)/2, (\sqrt{5}-1)/2, \infty$, respectively.
    \item An $I_4$ fibre over $t=1$ with components $L_2,\tilde{L}_2,E_{3,0}$ and the strict transform on $S$ of the
    pullback on $X$ of the plane curve $x-y=0$.
    \item An $I_2$ fibre over $t=-1$.
    \item Two $I_1$ fibres over the points satisfying $t^2+(2/9)t+(1/9)=0$.
\end{itemize}
\smallskip

\noindent (c) The group of points on $E_2(\QQ(t))$
generated by the $\mathcal{S}$-sections listed in the lemma (all images are defined over $\QQ(t)$) is isomorphic to $$\ZZ/2\ZZ\oplus\ZZ\; ,$$ 
where the first summand is generated by the $2$-torsion point $(0,0)$ and the free one is generated
by the point
$$ P_3=(4t^3,4t^3(t^2-1))\; . $$
\end{proposition}
\begin{proof}
Entirely analogous to the proof of Proposition \ref{p:EllFibOne}. Here we find the elliptic fibration
map
$$ S \longrightarrow {\PP}^1 \qquad [x:y:z:w] \mapsto [x(y+z):x^2+yz] $$
and use the $\PP^1$ parameter $t=x(y+z)/(x^2+yz)$ and the $E_{2,-1}$ section to give a $\QQ(t)$-point
on $S_t$.
Again, see~\cite{BFHN19} for computational details and explicit
transformation maps.
\end{proof}
\bigskip

\subsection{Third elliptic fibration}\label{ssec:Third}

\noindent The third example of Lemma~\ref{l:FibData} is a genus $1$ fibration with no section.
As shown in Section~\ref{sec:ComputationOfTheShiodaInoseStructureOfX}, however, this fibration provides a Shioda--Inose-type
structure that furnishes much useful arithmetic and geometric information about $S$.
\medskip

\begin{proposition}\label{p:EllFibThree}
(a) The generic fibre $S_t/\QQ(\sqrt{5})(t)$ of the genus $1$ fibration of Lemma \ref{l:FibData} (c)
is given by the quartic equation
\begin{align*}
    ty^2 = & \ x^4 + ((-116\sqrt{5}+272)t^2+(66\sqrt{5}-148)t-34\sqrt{5}+76)x^3 +\\
    & ((-23664\sqrt{5}+52974)t^4+(62037\sqrt{5}-138785)/2t^3+(-71882\sqrt{5}+160725)/2t^2+\\
    & (39297\sqrt{5} - 87871)/2t + (-3876\sqrt{5} + 8667)/2)x^2 + ((-2096932\sqrt{5} + 4689008)t^6 +\\ 
    & (8789895\sqrt{5} - 19655187)/2t^5 +(-14213809\sqrt{5} + 31783015)/2t^4 + \\
    & (14281062\sqrt{5} - 31933423)/2t^3 +(-10526810\sqrt{5} + 23538663)/2t^2 +\\
    & (5316367\sqrt{5} - 11887758)/2t + (-98209\sqrt{5} + 219602)/2)x+\\
    & ((-69643152\sqrt{5} + 155726921)t^8 + (191265401\sqrt{5} - 427682729)t^7 +\\
    & (-1317057443\sqrt{5} + 2945029977)/4t^6 + (2349501743\sqrt{5} - 5253645563)/8t^5 +\\
    & (-1901993416\sqrt{5} + 4252986577)/16t^4 + (19147095\sqrt{5} - 42814206)/4t^3 +\\
    & (-250668666\sqrt{5} + 560512177)/8t^2 + (610197963\sqrt{5} - 1364444125)/8t +\\
    & (-7465176\sqrt{5} + 16692641)/16)\; .
\end{align*}
\noindent (b) The full set of bad fibres for this fibration is given by the following fibres.
\begin{itemize}
    \item The $II^*$ fibres (i) and (ii) of Lemma~\ref{l:FibData}. 
    They lie over $t=\infty, 0,$ respectively.
    \item Four $I_1$ fibres over the points satisfying $t^4-(1118\sqrt{5} + 2598)/27t^3 -
    (89700\sqrt{5} + 200362)/27t^2 -(1118\sqrt{5} + 2598)/27t + 1=0$.
\end{itemize}
\end{proposition}
\begin{proof} Note that there is no fibration over $\QQ$ in this case since the
image of Kodaira fibre (i) $D$ under $\sqrt{5} \mapsto -\sqrt{5}$ is not a fibre
of the same fibration (it has non-zero intersection with $D$).

(a) The Riemann--Roch computation for Kodaira fibre (i) in Lemma~\ref{l:FibData} is much longer and harder in
this case than in the previous two. A  fibration map was returned of the form
$S \rightarrow \PP^1, \; (x:y:z:w)\mapsto (b_1: b_2)$,   where $b_1$ and $b_2$ are two degree 9 weighted polynomials in $x,y,z,w$ which
we do not write down here, but are in~\cite{BFHN19}.

As usual, letting $t=b_1/b_2$, we then computed a model for the generic fibre
$S_t$ of the fibration as a degree 10 plane curve $C$ over $\QQ(\sqrt{5})(t)$.
Using the degree 2 divisor $D$ on $C$ provided by $\tilde{L}_6$, and performing another
Riemann--Roch computation for $D$ on a non-singular embedding of $C$ in
$\PP^9$, we explicitly determined the 2-to-1 cover $S_t \rightarrow \PP^1$ corresponding to
$D$. Finally, a standard computation using differentials gave us
the equation for $S_t$ in the statement. This is laid out in~\cite{BFHN19}.
\medskip

(b) By the choice of $b_1, b_2$, the bad fibres over $0$ and $\infty$ are the two $II^*$ fibres. To
compute the other bad fibres, a $t=s^2$ substitution (giving a base change unramified
over $0,\infty$) allows the transformation to a Weierstrass cubic model over $\QQ (s)$ and standard application
of Tate's algorithm. This shows that the only other bad fibres of $S_t$ are $I_1$ fibres
at the stated points.
\end{proof}
\bigskip

\begin{remark}
Shioda--Inose structures associated with this fibration are made explicit in Section~\ref{sec:ComputationOfTheShiodaInoseStructureOfX}. We briefly include some extra information on that topic here.

There is a {\it Nikulin involution} $\iota$ \cite[Section 5]{Morrison_K3}, which is an
involution of $S$ over $\QQ(\sqrt{5})$ which swaps the two $II^*$ fibres and for which
the desingularised quotient $Y=S/\langle\iota\rangle$ is a Kummer surface.

From the explicit quartic equation of Proposition \ref{p:EllFibThree} (a), it is not too hard to show that $\iota$ is the
involution of $S$ associated with the isomorphism $\iota^*$ of the function field
$k(S)=k(t,x,y)$ ($k=\QQ(\sqrt{5})$)
$$ \iota^*: k(t,x,y)  \cong  k(t,x,y) $$
$$   t \mapsto 1/t \qquad x \mapsto \alpha x+\beta \qquad y \mapsto -\gamma y $$
where
$$ \alpha = t^{-2}\left(\frac{t+e}{et+1}\right) \qquad \beta = ((3035-1302\sqrt{5})/38)\left(\frac
  {(t-1)(t^2+ft+1)}{t^2(et+1)}\right) \qquad \gamma=t\alpha^2 $$
  with
$$ e = (138+67\sqrt{5})/19 \qquad f = (2770+1324\sqrt{5})/355 \, .$$

\noindent The following diagram commutes
$$ \begin{array}{rcl} S & \stackrel{\iota}{\longrightarrow} & S \\
  \downarrow && \downarrow \\
  \PP^1 & \stackrel{t \mapsto 1/t}{\longrightarrow} & \PP^1
  \end{array} $$
and $Y$ has a genus $1$ fibration with generic fibre over $k(s)$, $s=t+(1/t)-2$,
with quartic equation $sy_1^2=F(x_1)$ for a degree 4 monic polynomial $F$ over $k(s)$.
Here $x_1=x+\iota^*(x)$ and $y_1$ is an element of $k(s)$ times $y+\iota^*(y)$. We do not
write down the polynomial $F$ but it comes from the explicit computation of
$k(s,x_1,y_1)=k(t,x,y)^{\langle\iota^*\rangle}$. This computation and the explicit $F$
are in~\cite{BFHN19}. The surface
$Y$ is the minimal model over $\PP^1$ of this genus $1$ curve over $k(s)$.

\end{remark} 
\section{Computation of the Shioda--Inose structure} \label{sec:ComputationOfTheShiodaInoseStructureOfX}

In this section, we exhibit an explicit Shioda--Inose structure of the surface $S_{DY}$;
in doing so, we closely follow the exposition in~\cite{Morrison_K3} and~\cite{Naskrecki_hypergeometric}.

Let $X$ be any smooth algebraic surface over $\mathbb{C}$. 
The singular cohomology group $H^{2}(X,\mathbb{C})$ admits a Hodge decomposition
\[H^{2}(X,\mathbb{C})\cong H^{2,0}(X)\oplus H^{1,1}(X)\oplus H^{0,2}(X).\]
The N\'{e}ron--Severi group $\NS(X)$ of line bundles modulo algebraic equivalences naturally embeds into $H^{2}(X,\mathbb{Z})$ and can be identified with $H^{2}(X,\mathbb{Z})\cap H^{1,1}(X)$. 
This induces a structure of a lattice on $\NS(X)$. Its orthogonal complement in $H^{2}(X,\mathbb{Z})$ is denoted by $T_{X}$ and is called the \textit{transcendental lattice} of $X$. We denote by $\Lambda(n)$ the lattice with bilinear pairing $\langle\cdot, \cdot\rangle_{\Lambda(n)}=n\langle\cdot,\cdot\rangle_{\Lambda}$. Recall that for a K3 surface the notions of Picard group and N\'eron--Severi group coincide (cf.~\cite[Proposition 1.2.4]{Huy16}).

If $X$ is a K3 surface the lattice $H^{2}(X,\mathbb{Z})$ is isometric to the lattice $U^3\oplus E_{8}(-1)^2$ where $E_{8}(-1)$ denotes the standard $E_{8}$--lattice with opposite pairing, corresponding to the Dynkin diagram $E_{8}$. The lattice $U$ is the hyperbolic lattice which is generated by vectors $x,y$ such that $x^2=y^2=0$ and $x.y=1$.  Moreover, $\dim H^{2,0}(X)=1$. Any involution $\iota$ on $X$ such that $\iota^{*}(\omega)=\omega$ for a non--zero $\omega \in H^{2,0}(X)$ is called a \textit{Nikulin involution}.

It follows from \cite[Section 5]{Nikulin_Autom_K3} (see also \cite[Lemma 5.2]{Morrison_K3}) that every Nikulin involution has eight isolated fixed points and the rational quotient $\pi:X\dashrightarrow Y$ by a Nikulin involution gives a new K3 surface $Y$.

A given lattice $L$ has a \textit{Hodge structure} if $L\otimes\mathbb{C}$ has a Hodge decomposition, cf. \cite[Chapter 7]{Voisin_book_1}. There exists a \textit{Hodge isometry} between two lattices with a Hodge structure if they are isometric and the isometry preserves the Hodge decompositions, cf. \cite[Definition 1.4]{Morrison_K3}.

\begin{definition}[\protect{\cite[Definition 6.1]{Morrison_K3}}]
	A K3 surface $X$ admits a Shioda--Inose structure if there is a Nikulin involution on $X$ and the quotient map $\pi:X\dashrightarrow Y$ is such that $Y$ is a Kummer surface and $\pi_{*}$ induces a Hodge isometry $T_{X}(2)\cong T_{Y}$.
\end{definition}
Every Kummer surface admits a degree 2 map from an abelian surface $A$. It follows from \cite[Theorem 6.3]{Morrison_K3} that if $X$ admits a Shioda--Inose structure (Figure \ref{fig:Shioda_Inose})
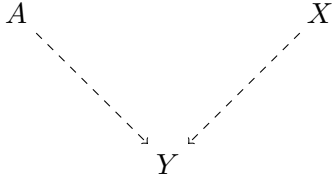
\begin{figure}[htb]
\centering
\begin{tikzpicture}

\node (v1) at (-2.5,1.5) {$A$};
\node (v3) at (1.5,1.5) {$X$};
\node (v2) at (-0.5,-0.5) {$Y$};
\draw[->]  (v1) edge[dashed] (v2);
\draw [->] (v3) edge[dashed] (v2);
\end{tikzpicture}
\caption{Shioda--Inose structure.}\label{fig:Shioda_Inose}
\end{figure}
then $T_{A}\cong T_{X}$. This follows from the fact that the diagram induces isometries $T_{A}(2)\cong T_{Y}$ and $T_{X}(2)\cong T_{Y}$. Alternatively, this is equivalent to the existence of an embedding $E_{8}(-1)^{2}\hookrightarrow \NS(X)$. 

\subsection{Shioda--Inose structure on the Drell--Yan K3 surface}

Let $S$ be the model of the Drell--Yan K3 surface introduced in
Section~\ref{sec:FullClassificationOfEllipticFibrationsOnX} in Proposition~\ref{p:EllFibThree}. 
The pullback of the generic fibre $S_{t}$ by the map $t\mapsto t^2$ produces a Kummer surface $\mathcal{K}$ with an explicit elliptic fibration $\mathcal{I}$. The fiber $\mathcal{I}_t$ above the point $t$ has equation
\[y^2=x^3+\frac{1}{6} \left(-45 \sqrt{5}-71\right) t^4 x+\frac{1}{2} \left(3-\sqrt{5}\right) t^8+\frac{1}{27} \left(-189 \sqrt{5}-551\right) t^6+\frac{1}{2} \left(3-\sqrt{5}\right) t^4.\]

\noindent Let $E(a,b)$ and $E(c,d)$ be the two elliptic curves defined by
\begin{align*}
    E(a,b)&\colon y^2=x^3+ax+b \; ,\\
    E(c,d)&\colon y^2=x'^3+c x'+ d \; .
\end{align*}
Consider the abelian surface $E(a,b)\times E(c,d)$ given by the product of the two elliptic curves defined above, 
and let $[-1]$ denote the automorphism of $E(a,b)\times E(c,d)$ given by multiplication by $-1$.
 Taking the quotient of  $E(a,b)\times E(c,d)$ by  $[-1]$, 
 we obtain a Kummer surface which has a natural elliptic fibration with parameter $u$:
\[x^3+ax+b-u^2(x'^3+cx'+d)=0.\]
This can be converted into the following  Weierstrass model, cf. \cite[\S 2.1]{Kumar_Kuwata}
\begin{equation}\label{eq:Kummer_fibr}
Y^2=X^3-3ac X+\frac{1}{64}(\Delta_{E(a,b)}u^2+864bd+\frac{\Delta_{E(c,d)}}{u^2})\; .
\end{equation}
The elliptic fibration $\mathcal{I}$ is isomorphic to \eqref{eq:Kummer_fibr}. Hence, we obtain the following system of equations:
\begin{align*}
A^2-5 &=0\; ,\\
1411985089 - 631459755A + 18ac &=0\; ,\\
131587540863282 - 58847737271814A + 108c^3 + 729d^2 &=0\; ,\\ 
-238992218766044 + 106880569389324A - 1458bd &=0\; ,\\ 
 131587540863282 + 108a^3 - 58847737271814A + 729b^2 &=0\; .
\end{align*}
Let $\mathcal{P}$ be the scheme defined by the above system of equations.
Let $K=\mathbb{Q}(\alpha,\beta)$ denote the number field where
$$\alpha=\sqrt{\frac{\sqrt{5}+1}{2}},\quad \beta=\sqrt[3]{\sqrt{2}-1}.$$

\begin{remark}
The field $K$ is isomorphic to $\mathbb{Q}[x]/\left( x^{24} - 24x^{18} - 18x^{12} + 24x^6 + 1 \right)$.
\end{remark}

The scheme $\mathcal{P}$ has exactly four $K$-rational points $\mathcal{P}_{i}=(a_i,b_i,c_i,d_i)$, $i=1,2,3,4$, $a_i,b_i,c_i,d_i\in K$. Each point $\mathcal{P}_{i}$ determines a pair of elliptic curves $E(a_i,b_i), E(c_i,d_i)$ and an abelian surface $A(\mathcal{P}_{i}) = E(a_i,b_i)\times E(c_i,d_i)$. For any two $i,j\in \{1,2,3,4\}$ there exists exactly one automorphism $\sigma_{i,j}:K\rightarrow K$ such that $A(\mathcal{P}_{i})$ is equal to the conjugate abelian surface $A(\mathcal{P}_{j})^{\sigma_{i,j}}$. Assume that $S$ admits a Shioda-Inose structure with the abelian variety $A(\mathcal{P}_j)$ for some $j\in\{1,2,3,4\}$. Thus $T_{S}$ is Hodge isometric to $T_{A(\mathcal{P}_j)}$. Since $S$ is defined over $\mathbb{Q}$ it is equal to all conjugates $S^{\sigma_{i,j}}$. Hence $T_{S}=T_{S^{\sigma_{i,j}}}\cong T_{A(\mathcal{P}_j)^{\sigma_{i,j}}}=T_{A(\mathcal{P}_i)}$ for every $i\neq j$.
It follows that the lattices $T_{A(\mathcal{P}_{i})}$ are Hodge isometric to each other. Hence, we fix one Shioda-Inose structure given by the following coordinates:
\begin{align*}
    A &=\sqrt{5}\, ,\\
    a &=\frac{1}{6} \left(10611 \sqrt{2}-18087 \sqrt{5}-4775 \sqrt{10}+40515\right) \beta\, ,\\
    b &=\frac{1}{27} \left(-4779461 \sqrt{5}+26 \sqrt{2} \left(113888-50921 \sqrt{5}\right)+10686297\right) \alpha\, ,\\
    c &=\frac{1}{6} \left(16 \left(832 \sqrt{5}-1869\right) \sqrt{2}+8537 \sqrt{5}-19293\right) \beta ^2\, ,\\
    d &=\frac{1}{27} \left(26 \left(50921 \sqrt{5}-113888\right) \sqrt{2}-4779461 \sqrt{5}+10686297\right) \alpha\, .
\end{align*}
Let $\mathcal{E}^{(d)}$ denote the quadratic twist by $d$ of an elliptic curve $\mathcal{E}$.
Let $E_{\mu,\nu}$ denote the elliptic curve given by
\[y^2=x^3+4x^2+2(1 -  4\mu\sqrt{2} - 3\nu\sqrt{5})x\; ,\]
with $\mu,\nu=\pm 1$.
The elliptic curve $E(a,b)$ is isomorphic to $E_{1}:=E_{1,1}$ and $E(c,d)$ is isomorphic to $E_{2}:=E_{-1,1}^{(-1)}$;
both isomorphisms are \textit{a priori} only defined over $\overline{\mathbb{Q}}$. Let $K_4=\mathbb{Q}(\sqrt{2},\sqrt{5})$.

\begin{proposition}\label{prop:model_compatibility}
The Kummer surface $\mathcal{K}=\Kum(E_{1},E_{2})$ attached to the abelian variety $E_{1}\times E_{2}$ is isomorphic  to the elliptic surface $\mathcal{I}$
over the quadratic extension $K_4(\eta)/K_4$,
where $\eta=\sqrt{117 \sqrt{2}+74 \sqrt{5}+37 \sqrt{10}+117}$.
\end{proposition}
\begin{proof}
The natural elliptic fibration on $\mathcal{K}$ is provided by the genus $1$ curve $$\mathcal{K}_{t}:x^3+4x^2+2(1 - 4\sqrt{2} - 3\sqrt{5})x-t^2(y^3-4y^2+2(1 + 4\sqrt{2} - 3\sqrt{5})y)=0.$$
It follows from a direct computation that the Weierstrass form of $\mathcal{K}_{t \cdot\eta}$ is isomorphic over $K_4(t)$ to $\mathcal{I}_{t}$.
\end{proof}
Let $E_{256.1-i2}$ denote the elliptic curve defined in
\cite[\href{https://www.lmfdb.org/EllipticCurve/4.4.1600.1/256.1/i/2}{Elliptic Curve 4.4.1600.1-256.1-i2}]{lmfdb}. Its Weierstrass equation is
\[E_{256.1-i2}: y^2=x^3+2 \left(\sqrt{2}+1\right)x^2+\frac{1}{2} \left(-10 \sqrt{2}-9 \sqrt{5}-6 \sqrt{10}-13\right)x.\]
The curve $E_{256.1-i2}$ is a quadratic twist of $E_{1,1}$ by the element $\kappa=\frac{1}{2}+\frac{1}{\sqrt{2}}$.
Let $\mathcal{L}_{1}$ denote the degree $8$ L-function over $\mathbb{Q}$ of the elliptic curve $E_{256.1-i2}$. Let $\rho$ denote the unique $2$-dimensional Artin representation of the field $\mathbb{Q}(1/\sqrt{\kappa})=\mathbb{Q}[x]/(-4 + 4 x^2 + x^4)$ and let $\mathcal{L}_{\rho}$ be the degree $2$ L-function over $\mathbb{Q}$ associated with $\rho$. Let $\mathcal{L}_{2}$ denote the degree $8$ L-function over $\mathbb{Q}$ associated with $E_{1,1}$. We denote by $L_{p}(\mathcal{L})$ the $p$-th Euler factor of the L-function~$\mathcal{L}$.

\begin{proposition}\label{prop:char_twist_of_E11}
For each prime $p\neq 2,5$ we have the equality
\[L_{p}(\mathcal{L}_{1}\otimes\mathcal{L}_{\rho},s) = L_{p}(\mathcal{L}_{2},s)^2.\]
\end{proposition}
\begin{proof}
The conclusion follows from the fact that both elliptic curves are related by a quadratic twist by $\kappa$. Hence, the tensor product of the L-function $\mathcal{L}_{1}$ by the Artin L-function $\mathcal{L}_{\rho}$ is equal to a square of the L-function of $E_{1,1}$ up to finitely many factors. We verify by a direct computation that those factors correspond to primes $p=2,5$.
\end{proof}

We are now ready to prove that $E_1$ is modular in two different ways, i.e. it corresponds to a certain Hilbert modular form and since it is a $\mathbb{Q}$-curve it also corresponds to a classical modular form over $\mathbb{Q}$. Since the curve $E_1$ is a twist of $E_{256.1-i2}$ and the latter curve has smaller conductor norm we explicitly prove the modularity of that curve instead. 

\begin{lemma}
The elliptic curve $E_{1}$ is a $\mathbb{Q}$-curve, i.e. it is isogenous over $\overline{\mathbb{Q}}$ to every Galois conjugate curve $E_{1}^{\sigma}$ for an automorphism $\sigma\in\Gal(\overline{\mathbb{Q}}/\mathbb{Q})$.
\end{lemma}
\begin{proof}
We have the following isomorphisms over $K_{4}$:
\begin{align*}
    E_{1,1} &=  E_1,   \hspace{1cm}  E_{-1,1} \cong E_{2}^{(-1)}\\
    E_{1,-1} &\cong G^{(2)}, \hspace{0.75cm} E_{-1,-1} \cong F^{(-2)}
\end{align*}
To prove the lemma it is enough to find an isogeny from $E_1$ to each curve $E_2$, $G$ and $F$.

Consider the map  $\phi:E_{1}\rightarrow E_{2}$ defined by
\begin{equation}
    \phi(x,y)=(\phi_x(x), \phi_y(x,y))
\end{equation} 
where
\begin{align*}
\phi_x(x):=& 
\frac{x \left(7 x^2+6 \left(\sqrt{2}+5\right) x-2 \sqrt{5} (x+3) \left(\sqrt{2} x+3\right)+54\right)}
{9 x^2-6 \left(3 \sqrt{5}+\sqrt{2} \left(\sqrt{5}+3\right)+1\right) x+18 \sqrt{5}+4 \sqrt{2} \left(5 \sqrt{5}+9\right)+74},\\
\phi_y(x,y):=&\frac{1}{D_y(x)} 
\bigg((63 x+142) x-\sqrt{5} (((11 x+23) x+34) x+72)+\\
& +\sqrt{2} \left(((17 x+38) x+82) x-2 \sqrt{5} ((9 x+19) x+12)+72\right)+192 \bigg) y\; ,\\
D_y(x):=&-27 x^3+27 \left(3 \sqrt{5}+\sqrt{2} \left(\sqrt{5}+3\right)+1\right) x^2+ \\
& -18 \left(9 \sqrt{5}+2 \sqrt{2} \left(5 \sqrt{5}+9\right)+37\right) x+\\
& +8 \left(54 \sqrt{5}+\sqrt{2} \left(32 \sqrt{5}+81\right)+95\right) \; .
\end{align*}    
  
The map $\phi$ is an isogeny of degree $3$;  
the kernel of $\phi$ is generated by a point with $x$-coordinate $1/3(\sqrt{5} + 3)\sqrt{2} + 1/3(3\sqrt{5} + 1)$.
 
The isomorphism $E_{1}[2](K_4)\cong\mathbb{Z}/2\mathbb{Z}$ implies that there exists a unique $2$-isogeny $\psi:E_{1}\rightarrow F$ over $K_4$. Similarly, since $E_{2}[2](K_4)\cong\mathbb{Z}/2\mathbb{Z}$, it follows that there exists a unique $2$-isogeny $\psi':E_{2}\rightarrow G$ over $K_4$ from $E_{2}$ with kernel $E_{2}[2]$.
\end{proof}

\begin{remark}
There exists also a $7$-isogeny from $E_{1,1}$ to the elliptic curve
\[\widetilde{E}:y^2=x^3+4 \left(18 \sqrt{10}+49\right) x^2+\left(-7888 \sqrt{2}-5046 \sqrt{5}+3528 \sqrt{10}+11282\right) x,\]
which is induced by the cyclic subgroup generated by the point with $x$-coordinate $(\sqrt{5} - 5)\sqrt{2} - 3\sqrt{5} + 3$. 
In total we have a cubic configuration of 2,3,7 isogenies, cf.~Figure \ref{fig:isogeny_graph}.

\begin{figure}[H]
    \centering
    \begin{tikzpicture}
   \matrix (m) [matrix of math nodes, row sep=3em,column sep=3em]{
    & \widetilde{E}& & * \\
    E_{1,1} & & E_{-1,1}^{(-1)} & \\
    & * & & * \\
    E_{-1,-1}^{(-2)} & & E_{1,-1}^{(2)} & \\};
  \path[-stealth]
    (m-1-2) edge (m-1-4) edge (m-2-1)
            edge [densely dotted] (m-3-2)
    (m-1-4) edge (m-3-4) edge (m-2-3)
    (m-2-1) edge [-,line width=6pt,draw=white] (m-2-3)
            edge (m-2-3) edge (m-4-1)
    (m-3-2) edge [densely dotted] (m-3-4)
            edge [densely dotted] (m-4-1)
    (m-4-1) edge (m-4-3)
    (m-3-4) edge (m-4-3)
    (m-2-3) edge [-,line width=6pt,draw=white] (m-4-3)
            edge (m-4-3);
\end{tikzpicture}
    \caption{Left-to-right maps: degree 3; top-to-bottom maps: degree 2; back-to-front maps: degree 7. A star denotes an explicit elliptic curve which can be computed from the given isogeny.}
    \label{fig:isogeny_graph}
\end{figure}
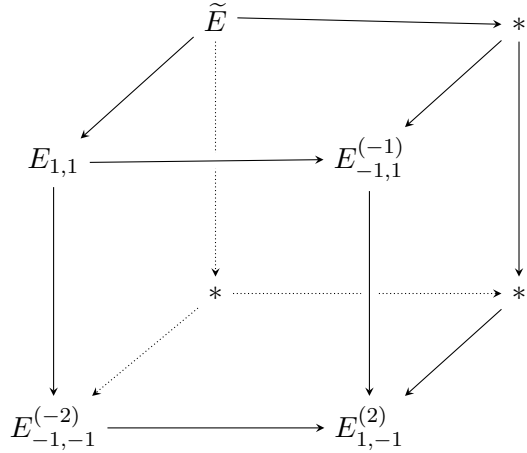
\end{remark}

An elliptic curve $\mathcal{E}$ over a totally real field $K$ is {\sl Hilbert modular} if there exists a Hilbert newform $f$ over $K$ of parallel weight $2$ and rational Hecke eigenvalues such that the $L$-functions $L(E,s)$ and $L(f,s)$ are equal.

\begin{lemma}
The elliptic curve $E_{256.1-i2}$ is Hilbert modular. The corresponding Hilbert modular form has conductor norm $256$ and is identified by the following label
\cite[\href{https://www.lmfdb.org/ModularForm/GL2/TotallyReal/4.4.1600.1/holomorphic/4.4.1600.1-256.1-i}{Hilbert form 4.4.1600.1-256.1-i}]{lmfdb}. 
\end{lemma}
\begin{proof}
The field $K_4(\zeta_5)$ is a quadratic extension of $K_4=\mathbb{Q}(\sqrt{2},\sqrt{5})$ where $\zeta_5$ is a primitive root of unity of degree $5$. The $5$-division polynomial of $E_{256.1-i2}$ is irreducible over $K_4(\zeta_5)$, hence the image of the modulo $5$ Galois representation $\rho=\overline{\rho}_{E_{256.1-i2},5}$ associated with $E_{256.1-i2}$ is not contained in the Borel subgroup and thus the image of $\rho$ is absolutely irreducible, cf. \cite[Prop. 2.1]{Freitas_modularity}. It follows from \cite[Thm.1]{Derickx_modularity} that the elliptic curve $E_{256.1-i2}$ is Hilbert modular. For the conductor norm $256$ there are exactly $9$ Hilbert newforms which could correspond to $E_{256.1-i2}$. A comparison of the L-series coefficients of $E_{256.1-i2}$ with those of the list of modular forms for several small primes reveals that the correct match is the form with a label \cite[\href{https://www.lmfdb.org/ModularForm/GL2/TotallyReal/4.4.1600.1/holomorphic/4.4.1600.1-256.1-i}{Hilbert form 4.4.1600.1-256.1-i}]{lmfdb}. 
\end{proof}

\begin{definition}
    A Hilbert modular form $\mathcal{H}$ defined over $F$ is a base change of a form $f$ defined over $E$ if the $L$-functions satisfy the condition
    \[L(\mathcal{H},s)=\prod_{\chi\in\Gal(F/E)^{\vee}}L(f\otimes\chi,s).\]
    \end{definition}
    \begin{remark}
    The definition of a base change is extracted from a general notion of a base change for $GL(2)$ forms, cf. \cite{Langlands_base_change}.
    \end{remark}
\begin{corollary}\label{c:Modular}
The Hilbert modular form $\mathcal{H}$ identified with the label \cite[\href{https://www.lmfdb.org/ModularForm/GL2/TotallyReal/4.4.1600.1/holomorphic/4.4.1600.1-256.1-i}{Hilbert form 4.4.1600.1-256.1-i}]{lmfdb} is a base change of the classical modular form $f$ of weight $2$ and level $160$ (identifier 
\cite[\href{https://www.lmfdb.org/ModularForm/GL2/Q/holomorphic/160/2/f/a/}{Newform 160.2.f.a}]{lmfdb}). In particular the Weil restriction $\Res_{\mathbb{Q}}^{K_{4}} E_{256.1-i2}$ of $E_{256.1-i2}$ is isogenous to the $\mathbb{Q}$-factor $A_f$ of the modular Jacobian $J(X_{1}(160))$ which corresponds to $f$. Moreover for every prime number $p$ and a prime ideal $\mathfrak{p}$ over $p$ it follows that $a_{\mathfrak{p}}(\mathcal{H})=a_{\sigma(\mathfrak{p})}(\mathcal{H})$ for every $\sigma\in\Gal(K_4/\mathbb{Q})$.
\end{corollary}
\begin{proof}
The Hilbert newform $\mathcal{H}$ has trivial character and is defined over $K_4=\mathbb{Q}(\sqrt{2},\sqrt{5})$, a biquadratic extension of $\mathbb{Q}$. Hence, if it came from a base change of a form $f$ without twist, the character of the form $f$ is of order at most $2$. The field $K_4$ is ramified only at $2$ and $5$ and the level norm of $\mathcal{H}$ is a power of $2$. The weight of the form $f$ is $2$. Assuming $K_4$ is a minimal splitting field, the dimension of the abelian variety attached to $f$ over $\mathbb{Q}$ is~$4$. The trace of each coefficient of the form $f$ is $4$ times the Hecke eigenvalue of $\mathcal{H}$. 
The primes $31,41,71$ and $79$ are totally split in $K_4$ and so it would follow that 
\begin{equation}\label{eq:traces_of_f}
    a_{31}(f)=-16,\quad a_{41}(f)=0,\quad a_{71}(f)=48,\quad a_{79}(f)=16.
\end{equation}
There exists a unique newform $f$ of level $160$ with character $(10/\cdot)$ and such that \eqref{eq:traces_of_f} holds.

The group of inner twists of the form $f$ is isomorphic to $C_2\times C_2$ and that implies the modular abelian fourfold $A_f$ attached to $f$ defined over $\mathbb{Q}$ is isogenous over $\overline{\mathbb{Q}}$ to a product $\prod_{\sigma\in\Gal(K_4/\mathbb{Q})} E_{256.1-i2}^{\sigma}$, cf. \cite{Gonzalez_Lario_Quer}. By base change of $f$ to $\mathcal{H}$, the elliptic curve $E_{256.1-i2}$ is modular over $K_{4}$. The conductor of the imprimitive L-function of $f$ is $2^{20}\cdot 5^4$ and the conductor of an $L$-function of $E_{1,1}$ over $K_4$ is $\Delta(K_4)^2\cdot Nm(\mathfrak{N})$ where $\mathfrak{N}$ is the level of the Hilbert modular form. By comparison we conclude that $Nm(\mathfrak{N})=256$. This restricts the search to $9$ isogeny classes of Hilbert modular forms and the computation of the eigenvalues for the primes of norm $9$ allows us to decide on the correct class.
\end{proof}

\begin{theorem}\label{t:SIStructure}
Let $E_1,E_2$ be the two elliptic curves defined over the field $K_{4}=\mathbb{Q}(\sqrt{2},\sqrt{5})$ by the following equations:
\begin{align*}
    E_1&: y^2=x^3+4x^2+2(1 - 4\sqrt{2} - 3\sqrt{5})x\, ,\\
    E_2&: y^2=x^3-4x^2+2(1 + 4\sqrt{2} - 3\sqrt{5})x\, .
\end{align*}
They are $3$-isogenous over $K_{4}$. There is a Shioda--Inose structure on $S$ with the Kummer surface $\Kum(E_1\times E_2)$. Let $p\geq 7$ be a prime number. We have that
\[|S(\mathbb{F}_{p})|=1+17p+\left(1+\left(\frac{5}{p}\right)\right)p+\mu(p)+p^2\, \]
with $\mu(p)=a_p(f)^2-\epsilon(p)p$ satisfying
where $\epsilon(p) = (\frac{10}{p})$ is a Kronecker quadratic character. 

Moreover, for $p\geq 7$ the number of points over $\mathbb{F}_{p^2}$ satisfies the formula

$$|S(\mathbb{F}_{p^2})| = 1+18p^2+t(p)^2+p^4,$$
where $t(p)$ is the trace of the Frobenius on $\mathbb{F}_{p^2}$ acting on the reduction of the curve $E_1$.
\end{theorem}
\begin{proof}
It follows from Proposition \ref{p:EllFibTwo} that there exists a basis of the N\'{e}ron--Severi group of $S$ in which all the elements of the basis are defined over $\mathbb{Q}$ except for the components not intersecting the zero section of the singular fibres above $t=\frac{1}{2}(-1\pm\sqrt{5})$. Under the action of an element 
$\sigma\in\Gal(\overline{\mathbb{Q}}/\mathbb{Q})$ such that $\sigma(\sqrt{5})=-\sqrt{5}$ the two components are permuted. Hence we conclude by Grothendieck--Lefschetz trace formula \cite[Appendix C \S 4]{Hartshorne:1977} that for a prime number $p$ of good reduction for $S$ we have
\[|S(\mathbb{F}_{p})| = 1+17p+\left(1+\left(\frac{5}{p}\right)\right)p+\mu(p)+p^2.\]
Let $H=H_{\mathbb{Q}_{\ell}}$ denote the orthogonal complement of the image of $\NS(S_{\overline{\mathbb{Q}}},{\mathbb{Q}_{\ell}})$ in $H^2_{et}(S_{\overline{\mathbb{Q}}},\mathbb{Q}_{\ell})$. For a prime $p\neq\ell$ of good reduction for $S$ there is a natural isomorphism $s:H^2_{et}(S_{\overline{\mathbb{Q}}},\mathbb{Q}_{\ell})\cong H^2_{et}(S_{\overline{\mathbb{F}}_{p}},\mathbb{Q}_{\ell})$. 
The number $\mu(p)$ is the trace of the Frobenius endomorphism $\Frob_{p}$ 
acting on the space $s(H)$ of dimension $3$.

From the existence of the Shioda--Inose structure on $S$ we know that the structure is determined by two elliptic curves $E_{a,b}$ and $E_{c,d}$. 
We find isomorphic (over $\overline{\mathbb{Q}}$) models of $E_{a,b}$ and $E_{c,d}$, respectively $E_{1}$ and $E_{2}$. 
It follows from Proposition \ref{prop:model_compatibility} that  the space $s(H)$ and $\Sym^2 H^1_{et}((E_1)_{\overline{\mathbb{F}}_{p}},\mathbb{Q}_{\ell})$ are isomorphic as $\Gal(\overline{\mathbb{F}}_{p}/\mathbb{F}_{p^4})$-modules. 
This follows from the existence of the $3$-isogeny between $E_{1}$ and $E_{2}$. 

The surface $S$ is isomorphic to the Inose fibration over $K_{8}=\mathbb{Q}(\sqrt{2},\sqrt{5},\eta)$ due to Proposition \ref{prop:model_compatibility}. The Galois group of the field $K_{8}$ is $C_{2}\times D_{4}$.
So, if the eigenvalues of the Frobenius $\Frob_{p^2}$ acting on $H^1_{et}((E_1)_{\overline{\mathbb{F}}_{p}},\mathbb{Q}_{\ell})$ are $\alpha,\beta$, then the eigenvalues of $\Frob_{p^2}$ acting on $\Sym^2 H^1_{et}((E_1)_{\overline{\mathbb{F}}_{p}},\mathbb{Q}_{\ell})$ are $\alpha^2,\beta^2,\alpha\beta$ and hence the trace of $\Frob_{p^2}$ on $\Sym^2 H^1_{et}((E_1)_{\overline{\mathbb{F}}_{p}},\mathbb{Q}_{\ell})$ is the same as the trace of $\Frob_{p^2}$ on $H$. Hence, the formula for the number of points in $S(\mathbb{F}_{p^2})$ follows.

The elliptic curves $E_1$ and $E_{2}$ are quadratic twists by $\frac{1}{2}+\frac{1}{\sqrt{2}}$ of $E_{256.1-i2}$ and $E_{256.1-i1}$, respectively. A Kummer surface $\Kum(\mathcal{E}\times\mathcal{F})$ associated to two elliptic curves $\mathcal{E}$, $\mathcal{F}$ is provided as a resolution of the double sextic $y^2=f_{\mathcal{E}}(x)f_{\mathcal{F}}(x')$ where $f_{\mathcal{E}}$, $f_{\mathcal{F}}$ are the cubic polynomials attached to the Weierstrass equation of $\mathcal{E}$ and $\mathcal{F}$. A simultaneous twist of $\mathcal{E}$ and $\mathcal{F}$ by an element $d$ provides a Kummer surface $\Kum(\mathcal{E}^{(d)}\times\mathcal{F}^{(d)})$ which is isomorphic to $\Kum(\mathcal{E}\times\mathcal{F})$ over the base field. It can be verified by a suitable change of coordinates.

Hence, we can replace in our considerations the product $E_{1}\times E_{2}$ with the product $E_{256.1-i2}\times E_{256.1-i1}$. Since the elliptic curves $E_{256.1-i1}$ and $E_{256.1-i2}$ are Hilbert modular and the corresponding form is a base change of the form $f$ it follows that $a_p(f)^2\in\mathbb{Z}$. Therefore the symmetric square motive $\mathcal{M}=Sym^2(M(f))$ of the modular motive $M(f)$ is defined over$\mathbb{Q}$ and has coefficients in $\mathbb{Q}$, cf. \cite[Theorem 1.2.4]{Scholl_motives}. Its $\ell$-adic realisation has trace $\mu(p)=a_p(f)^2-(\frac{10}{p})p$.
\end{proof}

\subsection{Supersingular reduction}\label{ssec:supersingular_reduction}

The Drell--Yan K3 surface has Picard rank 19 in characteristic $0$. When we reduce to characteristic $p$ the Picard rank can jumps to $20$ (ordinary reduction) or $22$ (supersingular reduction). We describe here under what conditions we have a supersingular reduction. 

Conjecturally, based on the Lang--Trotter heuristic \cite{Lang_Trotter}, \cite[Remark]{Elk89} the set of supersingular primes has density zero among all primes. However, the result of Elkies \cite{Elk89} proves that the there are infinitely many supersingular primes for the surface $S$. The sparseness of the set of supersingular primes provides a quantitative reason for why the proof of Proposition \ref{p:PicNumBound} was possible with a choice of two \textit{small} primes of non-supersingular reduction. In contrast, an argument such as that in Proposition \ref{p:PicNumBound} for a given K3 surface of Picard rank $20$ might require to use much larger prime numbers since the Picard rank jumps from $20$ to $22$ happen for a positive density of primes, cf. \cite[Theorem 1]{Shimada}.
\begin{corollary}
For primes $p$ such that $j\in \mathbb{F}_{p^2}$ is a supersingular $j$-invariant, we have that $\NS(S_{\overline{\mathbb{F}}_{p}})$ is of rank $22$, i.e., the prime $p$ is of supersingular reduction. The set of primes of supersingular reduction is infinite.
\end{corollary}
\begin{proof}
The rank $\rho(p)$ of the group $\NS(S_{\overline{\mathbb{F}}_{p}})$ is equal to $18+\rank\Hom(\tilde{E}_1,\tilde{E}_2)$ for a reduction modulo $p$ of the curves $E_{1}$, $E_{2}$, cf. \cite[\S 12.2.4]{Schutt_Shioda_book}. 
Since the curves $E_{1},E_{2}$ are linked by an isogeny and they do not have complex multiplication, it follows that $\rank\Hom(\tilde{E}_1,\tilde{E}_2)=2$ unless they have supersingular reduction at $p$.
Since $E_{1},E_{2}$ are defined over a field with at least one real embedding, it follows from \cite{Elk89} that there are infinitely many supersingular primes.
\end{proof}

\textbf{Supersingular primes computation:} To compute the primes of supersingular reduction in practice, we perform the following algorithm. First, we compute the minimal polynomial of the $j$-invariant of the curve $E_1$, namely
\[P(T)=T^4 - 6416768T^3 + 12470497280T^2 + 27021904707584T - 34447407894757376.\]
The elliptic curve $E_{1}$ has supersingular reduction at a prime ideal $\mathfrak{p}$ above a rational prime $p$ 
if the polynomial $P(T)$ modulo $p$ has a common root with the polynomial $S_{p}(T)=\prod_{j}(T-j)$, 
where the product is over supersingular $j$-invariants. 
The latter is computed effectively, cf. \cite[V]{silverman}, \cite{Finotti}. In fact, we checked all the odd primes $p$ smaller than $104729$ and the elliptic curve $E_{1}$ modulo $p$ is supersingular for the following values of $p$:
\begin{align*}
    &13, 29, 41, 113, 337, 839, 853, 881, 953, 1511, 1709, 1889, 2351, 3037, 3389, 4871, 5557,\\ 
    &5711,5741, 6719, 6733, 7237, 8821, 14489, 14869, 14951, 15161, 15791, 15973, 18229, 18257,\\ 
    &18313,18341, 20021, 21517, 23197, 24359, 26921, 27749, 28559, 33349, 33461, 33599, 34649,\\
    &37813, 40151,44101, 45389,47629, 49057, 50077, 50231, 52919, 54277, 54377, 58631, 60689,\\
    &64679, 65269, 68879, 69761, 70237, 70309, 72269,72911, 78791, 91309, 101501.
\end{align*}
\begin{remark}
It is worth pointing out that the explicit construction of a Shioda-Inose structure allows one to compute in practice the list of supersingular primes to a much higher bound than the approach through point counts discussed in Proposition \ref{p:PicNumBound}. In particular, our threshold of primes $p\leq 104729$ for the algorithm above becomes completely infeasible for the approach in Proposition \ref{p:PicNumBound}.
\end{remark}
\section{Computing the Picard lattice via elliptic fibrations} \label{sec:ComputingPicXViaEllipticFibrationsOnX}
This is the section on the computation of $\Pic \Sbar$ based on elliptic fibrations on $S$.
Recall that: $S$ is the desingularisation of the surface $X_{DY}\subset \PP(1,1,1,3)$ defined in~\eqref{eq:DY};
the map $\pi\colon S\to \PP^1$ is the elliptic fibration defined in Proposition~\ref{p:EllFibTwo};
we denote by $T$ the image in $\Pic S$ of a torsion section of $\pi$ (for example $(0,0)$), 
by $F$ the image of the general fiber $E_2$, and by $O$ the image of the zero section;
finally we denote by $N:=\NS(\Sbar)$ the geometric N\'eron--Severi group of~$S$.

\begin{remark}
As already noted on a K3 surface the notions of Picard group and N\'eron--Severi group coincide,
hence $\Pic \Sbar = N$.
In this section, we use the latter notion instead of the former, and we rely essentially upon the results contained in~\cite{Shi89}.
\end{remark}

Every singular fibre $\pi^{-1}(v)$ of the fibration $\pi$ has type $I_n$ and we order the components in a cyclic order, cf.~\cite{Shi89}, i.e., $\theta_{i}^{v}$ for $i=0,\ldots, n-1$, component $\theta_{v}^{i}$ intersects once the components $\theta_{v}^{i-1}$ and $\theta_{v}^{i+1}$ (enumeration modulo $n$). The component $\theta_{v}^{0}$ is the unique component that intersects the zero fibre.

\noindent The N\'{e}ron--Severi group of an elliptic surface is generated by the following divisors:
\begin{itemize}
    \item all components of the singular fibres, 
    \item images of sections which correspond bijectively to points in the Mordell--Weil group of the generic fibre.
\end{itemize}
Since the numerical and algebraic equivalence coincide on an elliptic surface \cite{Shi89}, it follows that it is enough for the N\'{e}ron--Severi group to consider the spanning set which contains only the components of the reducible fibres which do not intersect the zero component.
\begin{proposition}\label{p:NS}
	The N\'{e}ron--Severi group $N$ of $S$ is a lattice of rank $19$ and discriminant $24$. It is spanned by $P=P_3, T, F, O$, and the components of the singular fibres in fibration $\pi$ which do not intersect the zero section $O$ and lie above the following points: 
	
	\begin{itemize}
	    \item $t=0$: components $a_i=\theta_{t=0}^{i}$ for $i=1,\ldots, 9$;
	    \item $t=-1$: component $\theta_{t=-1}^{1}$;
	    \item $t=\frac{1}{2}(-1\pm\sqrt{5})$: components $\theta_{\pm}^{1}$;
	    \item $t=1$: components $b_{i}=\theta_{t=1}^{i}$, $i=1,2,3$.
	\end{itemize}
\end{proposition}
\medskip
The dual graph of the $-2$-curves which generate the N\'{e}ron--Severi group is represented in Figure~\ref{fig:NS_group_dual_graph}. We include for completeness also the component of the fibre above $t=\infty$ which is not used in the basis. Each edge $A-B$ represents a unique transversal intersection between curves $A$ and $B$.

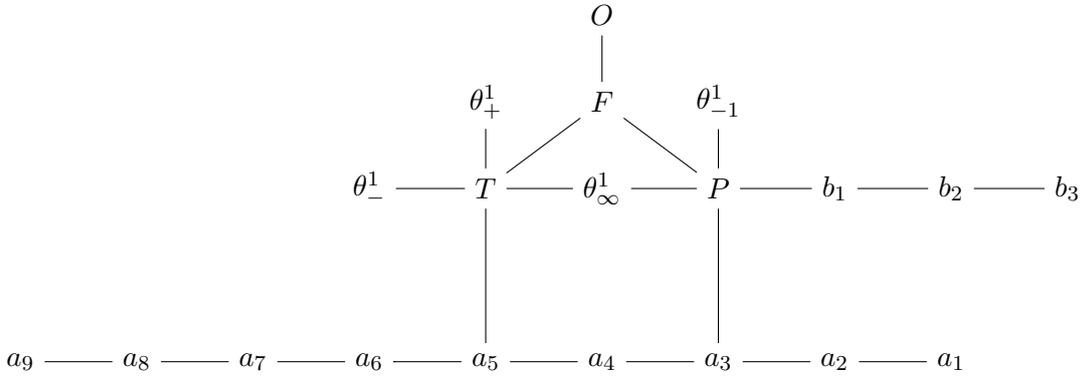
\begin{figure}[htb]
	\centering
	\begin{tikzpicture}[x=4em,y=3em]
	\node (a9) at (0,0){$a_9$};
	\node (a8) at (1,0){$a_8$};
	\node (a7) at (2,0){$a_7$};
	\node (a6) at (3,0){$a_6$};
	\node (a5) at (4,0){$a_5$};
	\node (a4) at (5,0){$a_4$};
	\node (a3) at (6,0){$a_3$};
	\node (a2) at (7,0){$a_2$};
	\node (a1) at (8,0){$a_1$};
	\node (b1) at (7,2){$b_1$};
	\node (b2) at (8,2){$b_2$};
	\node (b3) at (9,2){$b_3$};
	\node (m1) at (6,3){$\theta_{-1}^{1}$};
	\node (inf1) at (5,2){$\theta_{\infty}^{1}$};
	\node (eplus) at (4,3){$\theta_{+}^{1}$};
	\node (eminus) at (3,2){$\theta_{-}^{1}$};
	\node (T) at (4,2){$T$};
	\node (P) at (6,2){$P$};
	\node (O) at (5,4){$O$};
	\node (F) at (5,3){$F$};
	
	\draw (P)--(F);
	\draw (a9)--(a8)--(a7)--(a6)--(a6)--(a5)--(a4)--(a3)--(a2)--(a1);
	\draw (a5)--(T);
	\draw (a3)--(P);
	\draw (eminus)--(T)--(inf1)--(P)--(b1)--(b2)--(b3);
	\draw (eplus)--(T) -- (F) --(O);
	\draw (P) -- (m1);
	\end{tikzpicture}
    \caption{The dual graph of the $-2$-curves which generate the N\'{e}ron--Severi group.}
    \label{fig:NS_group_dual_graph}
\end{figure}

\subsection{A different proof}
In this subsection we prove Proposition~\ref{p:NS}.
For the convenience of the reader, we split the proof in four main steps, 
each corresponding to a subsection.
The computation of the rank still relies on Proposition~\ref{p:PicNumBound}, but not on the divisors exhibited in Section~\ref{sec:ComputingPicXViaDivisorsOnX};
the computation of the discriminant only relies on the elliptic fibration presented in Subsection~\ref{ssec:Second}.

The Shioda--Tate formula~\cite{Shi89} tell us that the rank of the group~$N$ is bounded from below by~$19$ and by Lefschetz theorem on (1,1)-classes \cite[Theorem 3.3.2]{Huy_Complex_geometry} it is bounded by~$20$ from above. 
To conlude that the rank of $N$ equals $19$ we reprove Proposition~\ref{p:PicNumBound} using an elliptic fibration on $S$. Let $p$ be a prime of good reduction for $S$. The number of points in $S(\mathbb{F}_{p^n})$ 
equals $G+B$ where $G$ is the total number of points on the elliptic curves over $\mathbb{F}_{p^n}$ in the fibres of good reduction and $B$ is the total number of points in the components defined over $\mathbb{F}_{p^n}$ of the fibres of bad reduction. This last step is done through a simple application of the Tate algorithm \cite{Tat75}. In our case it is enough to compute the numbers $|S(\mathbb{F}_{p^n})|$ for $n=1,2,3$ or $4$ to reconstruct the characteristic polynomial of the Frobenius morphism acting on the etale cohomology group $H^2_{et}(S_{\overline{\mathbb{F}_{p}}},\mathbb{Q}_{\ell})$ for $\ell\neq p$, cf.~\cite{vL07,Naskrecki_Acta,Naskrecki_BCP}.

\subsection{Height pairing computations}
Shioda \cite[Theorem 8.6]{Shi89} defined the quadratic positive semi-definite height pairing $\langle\cdot,\cdot\rangle$ on the group $E_2(\Qbar (t))$ which explicitly on the point $P$ is
\[\langle P, P\rangle = 4-\frac{a_0(10-a_0)}{10}-\sum_{i=1}^{4}\frac{a_i(2-a_i)}{2}-\frac{a_5(4-a_5)}{4},\]
where the correction values $a_i$, $a_{0}\in \{0,\ldots,9\}$, $a_1, a_2,a_3,a_4\in\{0,1\}$, $a_5\in\{0,1,2,3,4\}$ are determined from the intersection of $P$ with components of reducible fibres cf.\cite[p. 22]{Shi89}. It follows from the Tate algorithm \cite{Tat75}, \cite[IV,\S 9]{silverman_advanced} that for the point $P_3 = (4t^3,4t^3(t^2-1))$ the height $\langle P_3, P_3\rangle$ equals $3/20$. The minimal positive theoretically possible height of the point in $E_2(\Qbar(t))$ is equal to $1/20$ which follows from the height formula described above. The free part of $E_2(\Qbar(t))$ is of rank $1$. Hence if $P_3+\mathcal{T}$ were $m$-divisible for a suitable choice of a torsion point $\mathcal{T}$, then the height of the point $Q$ such that $mQ=P_3+\mathcal{T}$ would be equal to $\frac{3}{20m^2}<\frac{1}{20}$ for any $m\geq 2$, in contradiction to the minimality of height. Hence, the point $P_3$ spans the free part of $E_2(\Qbar(t))$.

\subsection{Discriminant formula}
As $E_2$ is the generic fibre of the elliptic fibration $S\to \PP^1$, 
the discriminant of $N$ can be computed from the discriminant formula, cf.\cite[\S 11.10]{Schutt_Shioda}
\[\disc N =(-1)^{r}\disc \Triv\cdot \disc \MW (S)/|\MW (S)_{\tors}|^2 \]
where $r$ is the rank of the group $E_{2}(\Qbar(t))$,
$\disc \Triv$ is the discriminant of the trivial sublattice with respect to the natural intersection pairing on $N$, 
$\disc \MW (S)$ is the discriminant of the lattice $E_{2}(\Qbar(t))/E_{2}(\Qbar(t))_{\tors}$ with respect to the height pairing $\langle\cdot,\cdot\rangle$ 
and $\MW(S)_{\tors}$ is $E_{2}(\Qbar(t))_{\tors}$. 
In our case we obtain $\disc N = 96/T^2$ where the integer $T\geq 1$ is the order of the torsion subgroup in $E_2(\Qbar(t))$. 
Since $\disc N$ is an integer, it follows that $T|2^2$. 
We have a unique point of order $2$ in $E_2(\Qbar(t))$ since the cubic polynomial which defines $E_2$ has only one root in $\Qbar(t)$. 
If there is a point $P_4$ of order $4$ on this curve, then $2P_4=(0,0)$. 
For a general point $(x,y)$ on $E_2$ the $x$-coordinate of the point $2(x,y)$ is
$$\frac{\left(16 t^7+16 t^6-16 t^5-x^2\right)^2}{4 x \left(16 t^7+16 t^6-16 t^5-3 t^4 x-8 t^3 x+2 t^2 x+x^2+x\right)}.$$
Hence if $(0,0)$ were $2$-divisible, 
the polynomial $x^2-16 t^5 \left(t^2+t-1\right)$ would have a root over $\overline{\QQ}(t)$, which is impossible.

Hence we conclude that the N\'{e}ron--Severi group $N$ is spanned by the components of the trivial sublattice (root sublattice generated by components of the fibres and the image of the zero section) and by the curve in $N$ representing $P_3$ and the torsion section $(0,0)$. 
Its discriminant is equal to $24$.
\begin{corollary}
It follows that
$$E_2(\Qbar(t)) = E_2(\QQ(\sqrt{5})(t))\cong \mathbb{Z}\oplus\mathbb{Z}/2\mathbb{Z}$$
and the group is generated by two points $P_3=(4t^3,4t^3(t^2-1))$ and $T=(0,0)$.
\end{corollary}

\subsection{N\'{e}ron--Severi group basis}
The group $N$ is spanned by the components of the reducible fibres, 
the general fibre $F$, 
the image of the zero section $O$ and the images of the non-zero sections which generate the Mordell--Weil group of the generic fibre. 
In our case, we have two points $P$ and $T$, where $P$ is of infinite order, and $T$ is a generator of the torsion subgroup ($2$-torsion point). 
We consider a generating set $\mathcal{B}$ for $N$, which contains only the following curves:
\begin{itemize}
    \item the components $\theta_{v}^{i}$ for $i>0$ of the reducible fibres (we skip the component which meets the zero section),
    \item the zero section $O$,
    \item the general fibre $F$,
    \item the sections $P$ and $T$.
\end{itemize}
The intersection pairing matrix for the tuple of curves above has dimension $20$ and rank~$19$. The curves satisfy the following linear relation:
\[a_1 + 2 a_2 + 3 a_3 + 4 a_4 + 5 a_5 + 4 a_6 + 3 a_7 + 2 a_8 + a_9  + \theta_{t=\infty}^{1} + \theta_{+}^{1} + \theta_{-}^{1}= 4 F +  2 O - 2 T,\]
where $a_i=\theta_{t=0}^{i}$ for $i\in\{1,\ldots,9\}$ and $\theta_{\pm}^{1}$ denotes the unique component which does not intersect zero in the fibre above $t=\frac{1}{2}(-1\pm\sqrt{5})$.
The set of components $\mathcal{B}_{0}=\mathcal{B}\setminus\{\theta_{t=\infty}^{1}\}$ is a basis of the N\'{e}ron--Severi group. 
Indeed, we check by a direct computation based on the intersection graph that the determinant of the sublattice spanned by $\mathcal{B}_{0}$ is~$24$.

We can also replace the generators $P$ and $T$ by $P-O-2F$ and $T-O-2F$, respectively, to obtain the following decomposition
\[N = L\oplus U,\]
where $L$ is positive definite of rank $17$ and discriminant $-24$ and $U$ is spanned by $F$ and $O$ and indefinite of rank $2$ and discriminant $-1$.

\begin{remark}
We checked with {\tt Magma} that the lattice $L$ is not a direct sum of proper sublattices. In the language of \cite{Shi89,Schutt_Shioda,Schutt_Shioda_book}, the lattice $L$ is the essential sublattice of $N$ with respect to the given elliptic fibration.
\end{remark}

\acknowledgments
The first author is very grateful to Claude Duhr, Lorenzo Tancredi, Robert M. Schabinger, Andreas von Manteuffel, Duco van Straten, and Stefan Weinzierl for lots of helpful discussions.
The second author was supported by the grant SFB/TRR 45 in Mainz.
The fourth author would like to thank John Voight for a helpful dicussion about Hilbert modular forms and the department of mathematics of Mainz University for the hospitality during his visit in February 2019.
We would like to thank the anonymous referees for their helpful comments. 

\bibliography{bib}

\end{document}